\documentclass[preprint]{elsarticle}

\journal{Journal Functional Analysis}

\usepackage{epsfig,amssymb,amsmath,amscd,amsfonts,color, amsthm,bbm}

\newtheorem{theo}{Theorem}[section] 
\newtheorem{lem}[theo]{Lemma}
\newtheorem{prop}[theo]{Proposition}

\newtheorem{cor}[theo]{Corollary}
\newtheorem{defi}[theo]{Definition}
\newtheorem{remark}[theo]{Remark}

\def\A{\mathcal A}
\def\build#1_#2^#3{\mathrel{\mathop{\kern 0pt#1}\limits_{#2}^{#3}}}
\def\C{\mathbb C }
\def\E{\mathbb E }
\def\epsilon{{\varepsilon}}
\def\F{\mathcal F}
\def\H{{H}}
\def\L{\mathcal L}
\def\M{\mathbb M}
\def\N{\mathbb N }
\def\phi{{\varphi}}
\def\qed{\hfill\hbox{\vrule\vbox to 2mm{\hrule width 2mm\vfill\hrule}\vrule}  \\}
\def\R{\mathbb R}
\def\ra{\to}
\def\SU{{\sf SU}}
\def\su{{\mathfrak{su}}}

\def\Tr{{\rm Tr}}
\def\tr{{\rm tr}}
\def\U{{\sf U}}
\def\UC{{\mathbb U}}
\def\u{\mathfrak{u}}
\def\Var{{\rm Var}}
\def\Z{\mathbb Z}

\def\t{\theta}

\begin{document}

\begin{frontmatter}

\title{Central limit theorem for the heat kernel measure on the unitary group}


\author[dma]{Thierry L\'evy}
\ead{levy@dma.ens.fr}

\author[orsay]{Myl\`ene Ma\"{\i}da \corref{cor1}}
\ead{mylene.maida@math.u-psud.fr}

\address[dma]{D\'epartement de Math\'ematiques, Ecole Normale Sup\'erieure, 45, rue d'Ulm, 
F-75230 Paris Cedex 05}

\address[orsay]{Laboratoire de Math\'ematiques,
Facult\'e des Sciences d'Orsay,
Universit\'e Paris-Sud,
F-91405 Orsay Cedex
}

\cortext[cor1]{Corresponding author. Postal address : Laboratoire de Math\'ematiques,
Facult\'e des Sciences d'Orsay,
Universit\'e Paris-Sud,
F-91405 Orsay Cedex. Tel: +33 1 69 15 77 95. Fax: +33 1 69 15 72 34}





\begin{abstract} We prove that for a finite collection of real-valued functions $f_{1},\ldots,f_{n}$ on the group of complex numbers of modulus $1$ which are derivable with Lipschitz continuous derivative, the distribution of $(\tr f_{1},\ldots,\tr f_{n})$ under the properly scaled heat kernel measure at a given time on the unitary group $\U(N)$ has Gaussian fluctuations as $N$ tends to infinity, with a covariance for which we give a formula and which is of order $N^{-1}$. In the limit where the time tends to infinity, we prove that this covariance converges to that obtained by P. Diaconis and S. Evans in a previous work on uniformly distributed unitary matrices. Finally, we discuss some combinatorial aspects of our results.
\end{abstract}



\end{frontmatter}





\section{Introduction}
In \cite{DE}, P. Diaconis and S. Evans studied the fluctuations of the trace of functions of a unitary matrix picked uniformly at random. Let us recall briefly their main result. If $U$ is a unitary matrix of size $N\geq 1$ and $f$ a real-valued function on the set $\UC$ of complex numbers of modulus $1$, then the eigenvalues  $\lambda_{1},\ldots,\lambda_{N}$ of $U$ belong to $\UC$ and $\tr f(U) =\frac{1}{N}\sum_{i=1}^{N} f(\lambda_{i}),$ where  $\tr$ is the normalized trace (so that $\tr(I_{N})=1$) and the matrix $f(U)$ is obtained from $U$ and $f$ by functional calculus. Using Weyl's integration formula and the rotational invariance of the Haar measure, it is easy to see that if $f:\UC\to \R$ is defined almost everywhere, is integrable
and has zero mean on $\UC$ then $\tr f(U)$ is defined for almost every $U$ and, seen as a random variable under the Haar measure, also has zero mean.


The function $f$ being fixed, $\tr f$ can be seen as a random variable on the unitary group $\U(N)$, endowed with the Haar measure, for all $N\geq 1$. Thus, the single function $f$ gives rise to a sequence of random variables indexed by the integer $N,$ which is their main object of study. In order to understand the behaviour of this sequence, a fundamental fact, which has been proved and used extensively in this context in \cite{DE}, is the following: for all $p,q\in \Z$, one has $\E[\tr(U^{p}) \overline{\tr(U^{q})}]=\delta_{p,q} N^{-2} \min(|p|,N)$. Using this, one can easily check that, if $f$ is square-integrable on $\UC$, then the variance of $\tr f$ converges to $0$ as $N$ tends to infinity. Moreover, if $f$ belongs to the Sobolev space $\H^{\frac{1}{2}}(\UC)$ (see Definition \ref{H1/2} below), then the series of the variances of $\tr f$ on $\U(N)$ converges, which gives a strong law of large numbers. 

The main result of \cite{DE} is that the fluctuations of $\tr f$ under the Haar measure are
asymptotically  Gaussian. More precisely, they have proved  that if $f$ belongs to $\H^{\frac{1}{2}}(\UC)$ and has zero mean on $\UC$, then $N \tr f$ converges in distribution to a centered Gaussian random variable with variance equal to the square of the $H^{\frac{1}{2}}$-norm of $f$ (see Theorem
\ref{tcl-haar} below for a precise statement). \\

In this paper, we consider the fluctuations of $\tr f$ when the unitary matrix is picked not under the Haar measure, but rather under the heat kernel measure at a certain time. The heat kernel measure at time $T$ is the distribution of $U_{N}(T)$, where $(U_{N}(t))_{t\geq 0}$ is the Brownian motion on $\U(N)$ issued from the identity matrix, that is, the Markov process whose generator is the Laplace-Beltrami operator associated to a certain Riemannian metric on $\U(N)$. The choice of a Riemannian metric that we make is explicited at the beginning of Section \ref{ubm-sec}. Apart from being one of the most natural stochastic processes with values in the unitary group, the Brownian motion arises for example in the context of two-dimensional $\U(N)$ Yang-Mills theory (\cite{KaKo,GW,GrMa}).  

Let $f:\UC\to \R$ be a function, as above. Once a time $T\geq 0$ is fixed, $\tr f$ is a random variable on $\U(N)$ for each $N\geq 1$, the unitary group being endowed with the heat kernel measure at time $T$. With our choice of Riemannian metric, it is known since the work of P. Biane \cite{Biane} that if $f$ is continuous, then $\tr f$ converges almost surely towards the integral of $f$ against a probability measure $\nu_{T}$ on $\UC$, which is characterized by the formula \eqref{nut} below. By this almost sure convergence, we mean that the expectations of these variables and the series of their variances converge.
For all $T>0$, the measure $\nu_{T}$ is absolutely continuous with respect to the uniform measure on $\UC$, with a density which unfortunately cannot be expressed in terms of usual functions. Its support is the full circle only for $T\geq 4$. For $T\in (0,4)$, its support is an arc of circle containing $1$, symmetric with respect to the horizontal axis, which grows continuously with $T$, and for the width of which a simple explicit formula exists. 
In fact, as $N$ tends to infinity, not only the distribution of the eigenvalues of $U_{N}(T)$ but the Brownian motion itself as a stochastic process converges in a certain sense towards a limiting object called the {\em free multiplicative Brownian motion}, which is defined in the language of free probability. The measure $\nu_{T}$ is the non-commutative distribution of this free process at time $T$ and can be considered as a multiplicative analogue of the Wigner semi-circle law. 

The main result of this paper is that for any function $f:\UC\to \R$ with Lipschitz continuous derivative, the fluctuations of $N \tr f$ are asymptotically Gaussian with variance $\sigma_{T}(f,f)$, where $\sigma_{T}$ is the quadratic form defined in Definition \ref{sigma}. This definition of $\sigma_{T}(f,f)$ involves three free multiplicative Brownian motions which are mutually free and the functional calculus associated to $f'$. It makes sense for functions of class $C^1$, or at best for absolutely continuous functions.  An alternative definition of $\sigma_{T}(f,f)$ is given by Definition \ref{newsigma} in terms of the Fourier coefficients of $f$ and the solution of an infinite triangular differential system (see Lemma \ref{freeito}). We prove that, when $T$ is large enough, this second definition makes sense for functions in the Sobolev space $H^{\frac{1}{2}}(\UC)$, which are not even necessarily continuous. 

Moreover, we prove that, as $T$ tends to infinity, $\sigma_{T}(f,f)$ converges towards the square of the $H^{\frac{1}{2}}$-norm of $f$. This convergence is consistent, at a heuristic level, with the result of P. Diaconis and S. Evans, since the Haar measure is the invariant measure of the Brownian motion, and its limiting distribution as time tends to infinity. 

For small values of $T$, the analysis seems much harder to perform. We have no expression of the covariance other than Definition \ref{sigma} and it seems plausible, considering the limiting support of the distribution of the eigenvalues of $U_{N}(T)$ and some puzzling numerical simulations (see Figure \ref{covariances} in Section \ref{diaco}), that the largest space of functions $f$ for which $N\tr f$ has Gaussian fluctuations might depend on $T$, say for $T\leq 4$. Unfortunately, we have no precise conjecture to offer in this respect.\\

The understanding of global fluctuations of random matrices has been widely developed in the literature using various techniques. By combinatorial methods applied to the computation of moments, Ya. Sinai and A. Soshnikov \cite{SS} derived a central limit theorem (CLT) for moments of Wigner matrices growing as $o(N^{2/3}).$ An important breakthrough is the work of  K. Johansson \cite{Joh} where he got, using techniques of orthogonal polynomials on the explicit joint density of eigenvalues, a CLT for Hermitian or real symmetric matrices whose entries have joint density $e^{N\tr V(M)}$, for a large class of potentials $V$. Recently, M. Shcherbina \cite{Shch} has been able to lower, in the symmetric case, the regularity of those functions for which the CLT holds. The study of Stieltjes transform for this purpose, initiated by L. Pastur and others \cite{pastur1, pastur2}, has recently given some striking results, among which one can cite the works of 
G. W. Anderson and O. Zeitouni \cite{AZ} or W. Hachem, P. Loubaton and J. Najim \cite{HLN}. Recently S. Chatterjee \cite{Chat} proposed ``a soft approach'' based on second order Poincar\'e inequalities.

The technique of proof that we have chosen is rather of the flavour of the one introduced 
in \cite{CD-fluc}. Therein, T. Cabanal-Duvillard proposed an approach based on matricial stochastic calculus to get a CLT for Hermitian and Wishart Brownian motions but also for several Gaussian Wigner matrices. In this direction we can also mention a CLT for band matrices obtained by A. Guionnet \cite{G-band}.

Some tools of free probability will play a key role in our analysis. The notion of second order freeness was developed in a series of papers \cite{sof1, sof2, sof3} in order to give a general framework to CLT's for large random matrices. In particular, the second paper \cite{sof2} of the series deals with unitary matrices and the results therein might be relevant to the problem under consideration
(see Section \ref{sec multitime} for more details).

Let us mention the work of F. Benaych-Georges \cite{Benaych}, which is closely related to  ours. He also considers  unitary matrices taken under the heat kernel measure, and he obtains a CLT for functions of the entries of these matrices, whereas we are rather considering functions of their empirical measure.\\

The paper is organized as follows : Section \ref{ubm-sec} is devoted to defining the Brownian motion on the unitary group, recalling from \cite{Biane} its  asymptotics, defining the proper covariance functional and stating our main result (Theorem \ref{main theo}). In Section \ref{proof-sec}, we present the structure of the proof of our main theorem by introducing a family of martingales (see Equation \eqref{mart}) that will be the main object of study. The proof will in fact boil down to proving the convergence of the bracket of these martingales (Section \ref{con bra}) and to controlling the variance of this bracket (Section \ref{con var}), relying on some technical results on the functional calculus on $\U(N)$ gathered in Section \ref{der fun cal}. In Section \ref{sec su}, we extend our result to other Brownian motions on the unitary group and to the Brownian motion on the special unitary group. In Section \ref{sec multitime}, we deal with the fluctuations of unitary Brownian motions stopped at different times. Section \ref{diaco} is devoted to the study of the covariance for large time, in connexion with the CLT for Haar unitaries \cite{DE}. Finally, in Section \ref{sec combi}, we discuss a combinatorial approach to some of our previous results and we obtain, via representation theoretic arguments, an explicit formula (Theorem \ref{covariance traces}) for mixed moments of the heat kernel on
$\SU(N)$.

\section{The Brownian motion on the unitary group}
\label{ubm-sec}
\subsection{The stochastic differential equation} Let $N\geq 1$ be an integer. We denote by $\U(N)$ the group of unitary $N\times N$ matrices and by $\u(N)$ its Lie algebra, which is the space of anti-Hermitian $N\times N$ matrices. We denote by $I_{N}$ the identity matrix. We will use systematically the following convention for traces: we denote the usual trace by $\Tr$ and the normalized trace by $\tr$, so that $\Tr(I_{N})=N$ and $\tr(I_{N})=1$.

Let us endow $\u(N)$ with the real scalar product $\langle X, Y \rangle_{\u(N)} = N \Tr(X^*Y)= - N \Tr(XY)$. We denote by $\|\cdot\|_{\u(N)}$ the corresponding norm.\\
The scalar  product $\langle \,\cdot\,, \,\cdot\, \rangle_{\u(N)}$ determines a Brownian motion with values in $\u(N)$, namely the
unique continuous Gaussian process $(K_{N}(t))_{t\geq 0}$ with values in $\u(N)$ such that
$$\forall s,t\geq 0, \forall A,B\in \u(N),\; \E[\langle A,K_{N}(s)\rangle_{\u(N)} \langle
B,K_{N}(t)\rangle_{\u(N)}]=\min(s,t) \langle A,B\rangle_{\u(N)}.$$

Equivalently, let $(B_{kl},C_{kl},D_{k})_{k,l\geq 1}$ be  independent standard real Brownian motions. Then $K_N(t)$ has the same distribution as the anti-Hermitian matrix whose upper-diagonal coefficients are the $\frac{1}{\sqrt{2N}}(B_{kl}(t)+ i C_{kl}(t))$ and whose diagonal coefficients are the $\frac{i}{\sqrt{N}} D_k(t)$.

The linear stochastic differential equation 
\begin{equation}\label{eds}
dU_{N}(t)= U_{N}(t) dK_{N}(t)-\frac{1}{2}U_{N}(t)dt
\end{equation}
admits a strong solution which is a process with values in $\M_{N}(\mathbb C)$. This process satisfies the identity $d(U_{N}U_{N}^*)(t)=0$, as one can check by using It\^{o}'s formula. Hence, this equation defines a Markov process on the unitary group $\U(N)$,
which we call the unitary Brownian motion. The generator of this Markov process can be described
as follows. Let $(X_{1},\ldots,X_{N^2})$ be an orthonormal basis of
$\u({N})$. Each element $X$ of $\u(N)$ can be identified with the left-invariant first-order
differential operator $\L_{X}$ on $\U(N)$ by setting, for all differentiable function $F:\U(N)\to 
\R$ and all $U\in \U(N)$,  
\begin{equation}\label{Lie}
(\L_{X}F)(U)=\frac{d}{dt}_{|t=0}F(U e^{tX}).
\end{equation}
The generator of the unitary Brownian motion is the second-order differential operator
$$\frac{1}{2}\Delta=\frac{1}{2}\sum_{k=1}^{N^2}\L_{X_{k}}^{2}.$$
This operator does not depend on the choice of the orthonormal basis of $\u(N)$. We denote the associated semi-group by $(P_{t})_{t\geq 0}$. From now on, we will always consider the Brownian motion issued from the identity matrix, so that $U_{N}(0)=I_{N}$.

The stochastic differential equation satisfied by $U_{N}$ can be translated into an It\^{o} formula, as follows. 

\begin{prop}  Let $F:\R\times \U(N)\to \R$ be a
function of class $C^2$. Then for all $t\geq 0$,
\begin{align}
F(t,U_{N}(t))=F(0,I_{N})&+\sum_{k=1}^{N^2} \int_{0}^{t} (\L_{X_{k}}F)(s,U_{N}(s))\; d\langle
X_{k},K_{N}\rangle_{\u(N)}(s) \nonumber\\
&+\int_{0}^{t} \left(\frac{1}{2}\Delta F+\partial_{t}F\right)(s,U_{N}(s))\; ds, \label{Ito}
\end{align}
and the processes $\{\langle X_{k},K_{N}\rangle_{\u(N)} :k\in\{1,\ldots,N^2\}\}$ are independent standard real
Brownian motions.
\end{prop}

This result is classical in the framework of stochastic analysis on manifolds (see for example \cite{IkedaWatanabe}), but since our whole analysis relies on this formula and for the convenience of the reader, we offer a sketch of proof in this particular setting.

\begin{proof} For all  $a,b\in \{1,\ldots,N\}$, let $\epsilon_{ab}:\M_N(\C)\to \C$ denote the coordinate mapping which to a matrix $M$ associates the entry $M_{ab}$. Let also $\partial_{ab}$ denote the partial derivation with respect to the $ab$-entry. The definition of $\L_X$ given by \eqref{Lie} makes sense for any matrix $X$. One can check the following identities:
\begin{multline*}
\forall X\in \M_N(\C) , \; \L_X=\sum_{a,b,c=1}^{N} \epsilon_{ac} X_{cb} \partial_{ab} \\
\mbox{ and } \L_{X}^2-\L_{X^2}=\sum_{a,b,c,a',b',c'=1}^{N} \epsilon_{ac} X_{cb} \epsilon_{a'c'}X_{c'b'} \partial_{ab}\partial_{a'b'},
\end{multline*}
\[ \Delta = \L_{C} + \sum_{k=1}^{N^2} \sum_{a,b,c,a',b',c'=1}^{N} \epsilon_{ac} (X_k)_{cb} \epsilon_{a'c'}(X_k)_{c'b'} \partial_{ab}\partial_{a'b'},\]
where $C=\sum_{i=1}^{N^2} X_i^2$. Moreover, $C=-I_N$, regardless of the choice of the orthonormal basis $(X_1,\ldots,X_{N^2})$.

Any smooth function $F:\R\times \U(N)$ is the restriction of a smooth function defined on $\R\times \M_N(\C)$. Applying the usual It\^{o} formula to this extended function and using the identities above leads immediately to \eqref{Ito}.
\end{proof}

\subsection{The free multiplicative Brownian motion} We are interested in the large $N$ behaviour of the stochastic process $U_{N}$ issued from $I_{N}$.
P.~Biane has described in \cite{Biane} the limiting distribution of this process seen as a collection of
elements of the non-commutative probability space $(L^{\infty}\otimes \M_{N}(\C),\E\otimes \tr)$. We start by describing the limiting object. As a general reference on non-commutative probability and freeness, we recommend \cite{VoiculescuSF}.

\begin{defi} Let $(\A,\tau)$ be a (non-commutative) $*$-probability space. A collection of
unitaries $(u_{t})_{t\geq 0}$ in $\A$ is called a {\em free multiplicative Brownian motion} if the
following properties hold.

1. For all $0\leq t_{1}\leq \ldots \leq t_{n}$, the elements
$u_{t_{1}},u_{t_{2}}u_{t_{1}}^{*},\ldots,u_{t_{n}}u_{t_{n-1}}^{*}$ are free.

2. For all $0\leq s\leq t$, the element $u_{t}u_{s}^{*}$ has the same distribution as $u_{t-s}$.

3. For all $t\geq 0$, the distribution of $u_{t}$ is the probability measure $\nu_{t}$ on
$\UC=\{z\in \C : |z|=1\}$ characterized by the identity
\begin{equation}\label{nut}
\int_\UC \frac{1}{1-\frac{z}{z+1} e^{tz} e^{\frac{t}{2}}\xi }\; d\nu_t(\xi)=1+z,
\end{equation}
valid for $z$ in a neighbourhood of $0$.
\end{defi}

The following result was proved by P. Biane. The second assertion follows from the first by a general result of D. Voiculescu.

\begin{theo}\label{cv mb mbul} The collection $(U_{N}(t))_{t\geq 0}$ of non-commutative random variables converges in distribution, as $N$ tends to $+\infty$, towards a free multiplicative Brownian motion.

Moreover, if $U_{N}^{(1)},U_{N}^{(2)},\ldots,U_{N}^{(n)}$ are $n$ independent sequences of unitary Brownian motions, then the family $((U_{N}^{(1)}(t))_{t\geq 0},(U_{N}^{(2)}(t))_{t\geq 0},\ldots,(U_{N}^{(n)}(t))_{t\geq 0})$ converges in non-commutative distribution, as $N$ tends to infinity, towards \linebreak
$((u^{(1)}_{t})_{t\geq 0},(u^{(2)}_{t})_{t\geq 0},\ldots,(u^{(n)}_{t})_{t\geq 0})$ where $u^{(1)},\ldots,u^{(n)}$ are $n$  free multiplicative Brownian motions which are mutually free.
\end{theo}

\subsection{Statement of the Central Limit Theorem} Recall that $\UC$ denotes the group of complex numbers of modulus $1$. Let $f:\UC\to \R$ be a function. Then, by the functional calculus, $f$ induces a function, still denoted by $f$, from $\U(N)$ to ${\mathbb M}_N(\C)$. Moreover, for all unitary matrix $U$, the matrix $f(U)$ is Hermitian.

We endow $\UC$ with the usual length distance, that is, the distance such that $d(e^{i\alpha},e^{i\beta})=|\alpha-\beta|$ for all $\alpha,\beta\in \R$ such that $|\alpha - \beta| \leq \pi$. Accordingly, we define the Lipschitz norm of a function $f:\UC \to \R$ as follows:
$$\|f\|_{\rm Lip}=\sup_{z,w\in \UC, z\neq w} \frac{|f(z)-f(w)|}{d(z,w)}.$$
Note that if $f$ is Lipschitz continuous and $z,w$ belong to $\UC$, then the following inequalities hold: $|f(z)-f(w)|\leq \| f \|_{\rm Lip} d(z,w)\leq \frac{\pi}{2} \|f\|_{\rm Lip} |z-w|$.

By the derivative of a differentiable function $f:\UC\to \R$, we mean the function $f':\UC\to \R$
defined by
$$\forall z\in \UC, \; f'(z)=\lim_{h\to 0} \frac{f(z e^{ih})-f(z)}{h}.$$
We denote by $L^1(\UC)$ the space of integrable functions on $\UC$, with respect to the Lebesgue measure. We denote by $C^1(\UC)$ the space of continuously differentiable functions and by $C^{1,1}(\UC)$ the subspace of $C^{1}(\UC)$ consisting of those functions whose derivative is Lipschitz continuous. We define a family of bilinear forms on $C^1(\UC)$ as follows.

\begin{defi}\label{sigma} Let $(\A,\tau)$ be a $C^*$-probability space which carries three free multiplicative Brownian motions $u,v,w$ which are mutually free. Let $T\geq 0$ be a real number. Let $f,g:\UC\to \R$ be two functions of $C^1(\UC)$. For all $s\in [0,T]$, we set
$\sigma_{T,s}(f,g)=\tau(f'(u_{s}v_{T-s})g'(u_{s}w_{T-s}))$. Then, we define
$$\sigma_{T}(f,g)=\int_{0}^{T} \sigma_{T,s}(f,g)\; ds = \int_{0}^{T} \tau(f'(u_{s}v_{T-s})g'(u_{s}w_{T-s}))\; ds.$$
\end{defi}

\begin{lem} For all $T\geq 0$, $\sigma_{T}$ is a symmetric non-negative bilinear form on $C^1(\UC)$.
\end{lem}

\begin{proof} The symmetry of $\sigma_{T}$ comes from the fact that the triples $(u,v,w)$ and $(u,w,v)$ have the same distribution. In order to prove the non-negativity, let us realize $(u,v,w)$ on the free product of three non-commutative probability spaces. So, let $({\mathcal A}_{u},\tau_{u})$, $({\mathcal A}_{v},\tau_{v})$ and $({\mathcal A}_{w},\tau_{w})$ be three non-commutative probability spaces which carry respectively $u$, $v$ and $w$. We consider their free product, so we define $\mathcal A=\mathcal A_{u}*\mathcal A_{v}*\mathcal A_{w}$ and $\tau=\tau_{u}*\tau_{v}*\tau_{w}$. We also use the notation $\tau_{u},\tau_{v},\tau_{w}$ for the partial traces on $\mathcal A$. Then
\begin{multline*}
\sigma_{T}(f,f)=\int_{0}^{T}\tau_{u}(\tau_{v}(f'(u_{s}v_{T-s}))\tau_{w}(f'(u_{s}w_{T-s})))\; ds\\
=\int_{0}^{T}\tau_{u}(\tau_{v}(f'(u_{s}v_{T-s}))^{2})\; ds \geq 0,
\end{multline*}
the positivity coming from the fact that $f'(u_s v_{T-s})$ is self-adjoint.
\end{proof}

We will use the notation $\sigma_{T}(f)=\sigma_{T}(f,f)$. Let us state our main result. 

\begin{theo}\label{main theo} Let $T\geq 0$ be a real number. Let $n\geq 1$ be an integer. Let $f_{1},\ldots,f_{n}:\UC\to \R$ be $n$ functions of $C^{1,1}(\UC)$. Let us define a $n\times n$ real non-negative symmetric matrix by setting  $\Sigma_T(f_1,\ldots,f_n)=(\sigma_{T}(f_{i},f_{j}))_{i,j\in\{1,\ldots,n\}}$. Then, as $N$ tends to infinity, the following convergence of random vectors in $\R^{n}$ holds in distribution:
\begin{equation}\label{main}
N\left(\tr f_{i}(U_{N}(T)) - \E\left[\tr f_{i}(U_{N}(T))\right]\right)_{i\in\{1,\ldots,n\}} \build{\longrightarrow}_{N\to\infty}^{(d)} {\mathcal N}(0, \Sigma_T(f_1,\ldots,f_n)).
\end{equation}
\end{theo}

\section{Structure of the proof}
\label{proof-sec}
For $T=0$, the result is straightforward. Let us choose once for all a real $T> 0$. In order to study the left-hand side of (\ref{main}), we write each component of this random vector as the difference between the final and the initial value of a martingale. To do this, let $(\F_{N,t})_{t\geq 0}$ denote the filtration generated by the unitary Brownian motion $U_{N}$.
To each function $f$ of $L^1(\UC)$ we associate a real-valued martingale $(M_{N}^f(t))_{t\in [0,T]}$ by setting
\begin{equation}
\label{mart}
M_{N}^f(t)=\E[\tr f(U_{N}(T))|\F_{N,t}].
\end{equation}

The left-hand side of (\ref{main}) is simply $N\left(M_{N}^{f_{i}}(T)-M_{N}^{f_{i}}(0)\right)_{i\in\{1,\ldots,n\}}$ and we are going to study the quadratic variations and covariations of the martingales $M_N^{f_{i}}$. In order to state the main technical results, let us introduce some notation.

Recall that the gradient of a differentiable function $F:\U(N)\to \C$ is the vector field on $\U(N)$ defined by $\nabla F=\sum_{k=1}^{N^2}(\L_{X_{k}}F)X_{k}$, where $(X_{1},\ldots,X_{N^{2}})$ is an orthonormal basis of $\u(N)$. To each pair of functions $f,g\in L^1(\UC)$ we associate a function $E_{N}^{f,g}$ on $[0,T)\times \U(N)$ by setting 
\[E_{N}^{f,g}(s,U)=N^2 \langle \nabla (P_{T-s}(\tr f))(U),\nabla (P_{T-s}(\tr g))(U) \rangle_{\u(N)}.\]
Let us check that this function is well-defined. By the Weyl integration formula, the fact that $f$ is integrable on $\UC$ implies that $\tr f$ is an integrable function on $\U(N)$. Hence, for all $s\in [0,T)$, $P_{T-s}(\tr f)$ is a function of class $C^{\infty}$ on $\U(N)$ and $E_N^{f,g}$ is well defined.

\begin{prop} \label{main technical} Consider $f,g\in L^1(\UC)$. With the notation introduced above, the following properties hold.
\begin{enumerate}
\item For all $t\in [0,T]$, the quadratic covariation of the martingales $NM^{f}_{N}$ and $NM^{g}_{N}$ is given by
$$\langle N M^{f}_{N}, N M^{g}_{N}\rangle_{t}=\int_{0}^{t}E_{N}^{f,g}(s,U_{N}(s))\; ds.$$
\item Assume that $f$ and $g$ are Lipschitz continuous. Then for all $s\in [0,T)$ and all $U\in \U(N)$, $| E_{N}^{f,g}(s,U)|\leq \left(\Vert f \Vert_{\rm Lip}+\Vert g \Vert_{\rm Lip}\right)^2$. Moreover, if $f$ and $g$ belong to $C^{1}(\UC)$, then the following convergence holds:
$$\E[E_{N}^{f,g}(s,U_{N}(s))]\build{\longrightarrow}_{N\to\infty}^{}\sigma_{T,s}(f,g).$$
\item Assume that $f$ and $g$ belong to $C^{1,1}(\UC)$. Then the following estimate holds:
$$\sup_{s\in [0,T)} \Var(E_{N}^{f,g}(s,U_{N}(s)))=O(N^{-2}).$$
\end{enumerate}
\end{prop}

Let us show that these results imply Theorem \ref{main theo}.

\begin{proof}[Proof of Theorem \ref{main theo}] For all $N\geq 1$, define a $\R^n$-valued martingale $Q_N=(Q^{1}_{N},\ldots,Q^{n}_{N})$ by setting $Q_N(t)=N \left(M_{N}^{f_{j}}(t)-M_{N}^{f_{j}}(0)\right)_{j\in \{1,\ldots,n\}}$. It is a martingale indexed by $[0,T]$, issued from $0$ and with the same bracket as \linebreak
$N\left(M_N^{f_{j}}\right)_{j\in \{1,\ldots,n\}}$. For all $\xi=(\xi_{1},\ldots,\xi_{n}) \in \R^n$ and all $t\in [0,T]$, set
$$R_{N}(t)=\exp\left(i \sum_{j=1}^{n} \xi_{j} Q_N^{j}(t) + \frac{1}{2}\sum_{j,k=1}^{n} \xi_{j}\xi_{k}  \int_0^t \sigma_{T,s}(f_{j},f_{k})\; ds \right).$$
It\^{o}'s formula yields
\begin{eqnarray*}
\E[R_{N}(t)]=1+\frac{1}{2}\sum_{j,k=1}^{n} \xi_{j}\xi_{k} \E  \int_0^t R_{N}(s) \left(\sigma_{T,s}(f_{j},f_{k})-E_{N}^{f_{j},f_{k}}(s,U_{N}(s))\right)\; ds.\\
\end{eqnarray*}
Thus, 
\begin{multline*}
\left| \E [R_{N}(t)- 1] \right| \leq 
\frac{n \Vert \xi\Vert^2}{2} e^{\frac{n T \Vert \xi \Vert^2}{2}  \build{\max}_{j=1\ldots n}^{} \Vert f'_{j} \Vert_\infty^2}\\ \max_{j,k=1\ldots n} \E\int_{0}^{t} \left| \sigma_{T,s}(f_{j},f_{k})-E_{N}^{f_{j},f_{k}}(s,U_{N}(s))\right|\; ds.
\end{multline*}
For fixed $j$ and $k$, the last integral is smaller than
\begin{multline*}
\int_{0}^{t} \left| \sigma_{T,s}(f_{j},f_{k}) - \E\left[E_{N}^{f_{j},f_{k}}(s,U_{N})(s)\right]\right| \; ds \\+ \E\int_{0}^{t} \left| E_{N}^{f_{j},f_{k}}(s,U_{N})(s)  - \E\left[E_{N}^{f_{j},f_{k}}(s,U_{N})(s)\right]\right| \; ds.
\end{multline*}
By the second part of Proposition \ref{main technical}, and by the dominated convergence theorem, the first integral tends to $0$ as $N$ tends to infinity. The square of the second integral is smaller than $t\int_{0}^{t} \Var(E_{N}^{f_{j},f_{k}}(s,U_{N}(s)))\; ds$, which, thanks to the third part of Proposition \ref{main technical} and by dominated convergence again, tends also to $0$. 
Finally, we have proved that 
$$\forall \xi \in \R^{n}, \; \lim_{N\to \infty} \E \left[e^{i \sum_{j=1}^{n}\xi_{j} Q_N^{j}(t)}\right]=\exp \left(- \frac{1}{2} \sum_{j,k=1}^{n} \xi_{j} \xi_{k} \int_0^t \sigma_{T,s}(f_{j},f_{k})\; ds\right),$$
which, for $t=T$, yields the expected result.
\end{proof}

In Section \ref{der fun cal}, we collect some technical results that we use in Sections \ref{con bra} and \ref{con var} to prove Proposition \ref{main technical}.

\section{Regularity of the functional calculus} \label{der fun cal}
In this section, we relate the regularity of a function $f:\UC \to \R$ to the regularity of the functional calculus mapping $f:\U(N)\to \M_{N}(\C)$ and the function $\tr f:\U(N)\to \R$. We start with a result which, logically speaking, is not necessary for our exposition, but which is the simplest instance of a crucial phenomenon.

\subsection{Lipschitz norms}
The group $\U(N)$ becomes a metric space when it is endowed with the Riemannian distance, denoted by $d$, associated to the Riemannian metric induced by the scalar product $\langle \cdot , \cdot \rangle_{\u(N)}$ on $\u(N)$. We denote by $\Vert F \Vert_{\rm Lip}$ the corresponding Lipschitz norm of a function $F:\U(N)\to \R$, that is,
$$\Vert F \Vert_{\rm Lip} = \sup\left\{ \frac{|F(U)-F(V)|}{d(U,V)} : U,V \in \U(N), U\neq V \right\}.$$

As a reference for the notions of Riemannian geometry that we use, we recommend \cite{DoCarmo}.

\begin{prop}\label{norme lip} Let $f:\UC \to \R$ be a Lipschitz continuous function. Then $\tr f:\U(N)\to \R$ is also Lipschitz continuous and
$$\| \tr f \|_{\rm Lip} = \frac{1}{N} \| f \|_{\rm Lip}.$$
\end{prop}

Note that this result can be compared to Lemma 1.2 in \cite{GZconc}, where it was a key point
towards the concentration results for Wigner and Wishart random matrices.
In order to prove this proposition, we use the following lemma.

\begin{lem}\label{codiag} Let $U$ and $V$ be two elements of $\U(N)$. Then there exists $A,B \in \U(N)$ such that $AUA^{-1}$ and $BVB^{-1}$ are diagonal  and $d(AUA^{-1},BVB^{-1})\leq d(U,V)$.
\end{lem}

\begin{proof} Let $\mathcal O$ be the conjugacy class of $V$. It is a compact submanifold of $\U(N)$. Let $V'$ be a point of $\mathcal O$ which minimizes the distance to $U$. Let $\gamma:[0,1]\to \U(N)$ be a minimizing geodesic path from $V'$ to $U$ parametrized at constant speed. It is thus of the form $\gamma(t)=V'e^{tZ}$ for some $Z\in \u(N)$. Since $V'$ minimizes the distance to $U$, the vector $\dot\gamma(0)$ is orthogonal to the tangent space $T_{V'}\mathcal O$. This space $T_{V'}\mathcal O$, identified with a subspace of $\u(N)$ by a left translation, is the range of the linear mapping ${\rm Ad}(V'^{-1})-{\rm Id}$. Hence, $Z$ belongs to the kernel of the adjoint linear mapping, that is, to the kernel of ${\rm Ad}(V')-{\rm Id}$. In other words, $V'ZV'^{-1}=Z$. It follows that $Z$ and $V'$ can be simultaneously diagonalized, in an orthonormal basis, and the same is true for $V'$ and $V'e^{Z}=U$. Finally, $V'$ and $U$ are conjugated by a same unitary matrix to two diagonal unitary matrices. The result follows easily from the fact that translation are isometries on $\U(N)$.
\end{proof}

\begin{proof}[Proof of Proposition \ref{norme lip}] Let $f:\UC \to \R$ be Lipschitz continuous. Consider $U$ and $V$ in $\U(N)$. Thanks to Lemma \ref{codiag}, let us choose $U'$ and $V'$ which are both diagonal, conjugated respectively to $U$ and $V$, and such that $d(U',V')\leq d(U,V)$. Let us write $U'={\rm diag}(e^{i\alpha_{1}},\ldots,e^{i\alpha_{N}})$ and $V'={\rm diag}(e^{i\beta_{1}},\ldots,e^{i\beta_{N}})$ in such a way that $|\beta_{j}-\alpha_{j}|\leq \pi$ for all $j\in \{1,\ldots,N\}$. Let us compute $d(U',V')$. It is equal to $d(I_{N},U'^{-1}V')$, hence to
\begin{multline*}
d(I_N,e^{i {\rm diag}(\beta_{1}-\alpha_{1},\ldots,\beta_{N}-\alpha_{N})})=\Vert i {\rm diag}(\beta_{1}-\alpha_{1},\ldots,\beta_{N}-\alpha_{N}) \Vert_{\u(N)}\\ =\sqrt{N \sum_{j=1}^{N} (\beta_{j}-\alpha_{j})^{2}}.
\end{multline*}
It follows that $d(U,V)\geq \sum_{j=1}^{N}|\beta_{j}-\alpha_{j}|$. On the other hand,
\begin{multline*}
|\tr f(V)-\tr f(U)| \leq \frac{1}{N} \sum_{j=1}^{N}|f(e^{i\beta_{j}})-f(e^{i \alpha_{j}})|\leq \frac{1}{N} \|f\|_{\rm Lip} \sum_{j=1}^{N}|\beta_{j}-\alpha_{j}|\\ \leq \frac{1}{N} \|f\|_{\rm Lip} d(U,V).
\end{multline*}
This proves the inequality $\| \tr f \|_{\rm Lip}\leq \frac{1}{N}\| f \|_{\rm Lip}$. By choosing $\alpha,\beta$ such that $|f(e^{i\beta})-f(e^{i\alpha})|$ is close to $\| f \|_{\rm Lip} |\beta-\alpha|$ and by considering $U=e^{i\alpha}I_{N}, V=e^{i\beta}I_{N}$, one verifies that the opposite inequality holds.
\end{proof}

Let us make a short heuristic comment on this result. The scalar product which we have chosen on $\u(N)$ corresponds to a metric structure on $\U(N)$ which gives this group the diameter $d(I_{N},-I_{N})=\| i{\rm diag}(\pi,\ldots,\pi)\|_{\u(N)}=N\pi$, of the order of $N$. The function $f:\UC\to \R$ being fixed, the variations of the function $\tr f:\U(N)\to \R$ are of the same order of magnitude as those of $f$ but occur on a space $N$ times as large. This makes the equality that we have juste proved plausible.

In the same order of ideas, note that the distance to the origin at time $T$ of a linear Brownian motion in a Euclidean space of large dimension $d$ is, by the law of large numbers, of the order of $\sqrt{dT}$. Assuming that the Brownian motion on the unitary group behaves in a comparable way, and considering the fact that the dimension of $\U(N)$ is $N^{2}$, this indicates that the Brownian motion $U_{N}(T)$ might be at a distance of order $N\sqrt{T}$ of $I_{N}$, thus a fraction of the diameter of $\U(N)$ which does not depend on $N$. This gives an intuitive justification for the choice of the normalization.

\subsection{First derivatives}
We are now going to prove that the functional calculus induced by $f$ is differentiable when $f$ is differentiable, and to compute its differential. For this, we introduce some notation. Let $f:\UC\to \C$ be a differentiable function. Let us define a function ${\sf D}f:\UC \times \UC \to \C$ by setting
$$\forall z,w\in \UC, \; {\sf D}f(z,w)=\left\{\begin{array}{cl} \frac{f(z)-f(w)}{z-w} & \mbox{
if } z\neq w,\\
-\frac{i}{z} f'(z)& \mbox{ if } z=w.\end{array}\right.$$
The function ${\sf D}f$ is symmetric and, if $f$ is $C^1(\UC)$, it is continuous and bounded by $\frac{\pi}{2}\Vert f' \Vert_{\infty}$. Note that ${\sf D} f$ takes its values in $\C$ even if $f$ is real-valued.

If the function $f$  is only Lipschitz continuous, then it is differentiable with bounded differential outside a negligible subset of $\UC$, and the definition of ${\sf D}f$ still makes sense outside the corresponding  negligible subset of the diagonal of $\UC\times \UC$. Moreover, outside this subset, the inequality $|{\sf D}f(z,w)|\leq \frac{\pi}{2} \| f'\|_\infty$ holds.

If $U$ is a unitary matrix, we denote by $L_{U}$ and $R_{U}$ the linear operators on ${\mathbb M}_{N}(\C)$ of left and right multiplication by $U$ respectively. These operators commute and they are normal with respect to the scalar product $\langle A,B\rangle = N \Tr(A^*B)$ on ${\mathbb M}_{N}(\C)$. In fact, $L_U^*=L_{U^{-1}}$ and $R_U^*=R_{U^{-1}}$. Hence, if $g$ is a function on $\UC\times \UC$, then $g(L_{U},R_{U})$ is a well-defined endomorphism of $\M_{N}(\C)$. Even when $f$ is only Lipschitz continuous, ${\sf D}f(L_U,R_U)$ is well-defined for almost all $U\in \U(N)$.

Let us define a special orthonormal basis of $\u(N)$. We use the notation $(E_{jk})_{j,k\in \{1,\ldots,N\}}$ for the canonical basis of $\M_N(\C)$. For all $j,k$ with $1\leq j <k \leq N$, set
$X_{jk}=\frac{1}{\sqrt{2N}} (E_{jk}-E_{kj})$ and
$Y_{jk}=\frac{i}{\sqrt{2N}} (E_{jk}+E_{kj})$. For all $j\in \{1,\ldots,N\}$, set
$H_j=\frac{i}{\sqrt{N}} E_{jj}$. These matrices form an orthonormal basis of $\u(N)$.

\begin{prop}\label{derivee f} Let $f:\UC \to \C$ be a differentiable function. Let $U$ be an element of $\U(N)$.
Let $X$ be an element of $\u(N)$. Then
\begin{equation}\label{der df}
\frac{d}{dt}_{|t=0} f(U e^{t X})= \left({\sf D}f(L_{U},R_{U})\right)(UX).
\end{equation}
In particular, when $U$ is a diagonal matrix with
diagonal coefficients $(u_1,\ldots,u_N)$, the following equalities hold.
\begin{enumerate}
\item For all $j\in \{1,\ldots,N\}$, $\frac{d}{dt}_{|t=0} f\left(U e^{t H_j}\right)={\sf D}f(u_{j},u_{j})UH_{j}$.
\item For all $j,k\in \{1,\ldots,N\}$ with $j<k$, $\frac{d}{dt}_{|t=0} f\left(U e^{t X_{jk}}\right)={\sf D}f(u_{j},u_{k})UX_{jk}$.
\item For all $j,k\in \{1,\ldots,N\}$ with $j<k$, $\frac{d}{dt}_{|t=0} f\left(U e^{t Y_{jk}}\right)={\sf D}f(u_{j},u_{k})UY_{jk}$.
\end{enumerate}
If $f$ is only Lipschitz continuous, then the same conclusions hold for almost all $U\in \U(N)$ (with respect to Haar measure).
\end{prop}

\begin{proof} We will give the proof under the assumption that $f$ is differentiable. The extension to the Lipschitz continuous case is straightforward (We have to take into account that in this case the differential operators 
involved are only defined for almost all $U$ with respect to Haar measure).\\

 Let us start by proving the part of the statement which concerns a diagonal matrix $U$.

1. Since $U e^{t H_j}$ is diagonal,
this assertion is proved by an easy direct computation.

2. This case is less trivial. Let us assume that $u_j\neq u_k$. Then for small $t$, there is a unique pair of continuous functions $(u_j(t),u_k(t))$ such that the spectrum of $U e^{tX_{jk}}$ is deduced from that of $U$ by replacing $u_j$ and $u_k$ respectively by $u_j(t)$ and $u_k(t)$. The functions $u_j$ and $u_k$ are in fact smooth and they satisfy $u_j'(0)=u_k'(0)=0$, an equality which can be phrased by saying that the right multiplication by $e^{tX_{jk}}$ does not affect the spectrum of $U$ at the first order. 

Let $D(t)$ be the diagonal matrix obtained from $U$  by replacing $u_j$ and $u_k$ by $u_j(t)$ and $u_k(t)$ respectively. By diagonalizing $U e^{tX_{jk}}$ for small $t$, one can find a unitary matrix $P(t)$ which depends smoothly on $t$, such that $P(0)=I_N$, such that the only non-zero off-diagonal terms of $P(t)$ are $P(t)_{jk}$ and $P(t)_{kj}$, and finally such that
\begin{equation}\label{diagonalise}
U e^{tX_{jk}}=P(t) D(t) P(t)^{-1}.
\end{equation}
By differentiating with respect to $t$ at $t=0$, one finds
$$UX_{jk}=[P'(0),U],$$
from which one deduces that $P'(0)_{jk}=\frac{1}{\sqrt{2N}}\frac{u_j}{u_k-u_j}$ and $P'(0)_{kj}=\frac{1}{\sqrt{2N}}\frac{u_k}{u_k-u_j}$.
By applying $f$ to both sides of (\ref{diagonalise}) and then differentiating again with respect to $t$ at $t=0$, we find
$$\frac{d}{dt}_{|t=0} f\left(U e^{t X_{jk}}\right)=[P'(0),f(U)].$$
Knowing the off-diagonal terms of $P'(0)$ is enough to compute this bracket and we find the expected result.
The case where $u_j=u_k$ is left to the reader, as well as the third assertion.

Let us now turn to the first part of the statement, where no assumption is made on $U$. Let us
first prove that (\ref{der df}) is true when $U$ is a diagonal matrix with diagonal coefficients
$(u_{1},\ldots,u_{N})$.

In this case, for all $j,k\in \{1,\ldots,N\}$, the matrix $E_{jk}$ is an eigenvector for $L_{U}$ and $R_{U}$, with the eigenvalues
$u_{j}$ and $u_{k}$ respectively. Hence, by definition of ${\sf D}f$, $E_{jk}$ is an eigenvector of
${\sf D}f(L_{U},R_{U})$ with the eigenvalue ${\sf D}f(u_{j},u_{k})$.  The validity of (\ref{der df}) in this case follows, because
$UH_{j}$ (resp. $UX_{jk}$, $UY_{jk}$) has the same vanishing entries as $H_{j}$ (resp. $X_{jk}$, $Y_{jk}$).

Let us finally prove that (\ref{der df}) holds for any unitary matrix. Consider $U\in \U(N)$. Choose
$P,D\in \U(N)$ such that $D$ is diagonal and $U=PDP^{-1}$. Set $Y=P^{-1}XP$. Then $U e^{tX}=P D
e^{tY} P^{-1}$. The result follows now easily.
\end{proof}

Before we apply the last result in order to compute the differential of $\tr f$, let us state a classical yet very useful lemma, of which a version can be found in \cite{SenguptaMF}.

\begin{lem}\label{trace} Let $(X_k)_{k\in \{1,\ldots,N^2\}}$ be a orthonormal basis of $\u(N)$. Let
$A,B$ be elements of $\M_N(\C)$. Then the following equalities hold:
\begin{equation}
\label{coal}
\sum_{k=1}^{N^2} \tr(AX_k) \tr(BX_k) =-\frac{1}{N^2} \tr(AB),
\end{equation}
\begin{equation}\label{split}
\sum_{k=1}^{N^2} \tr(AX_k BX_k)=-\tr(A) \tr(B).
\end{equation}
\end{lem}

\proof 1. For $A,B\in \u(N)$, this equality multiplied by $N^4$ is indeed simply 
$$\sum_{k=1}^{N^2} \langle A,X_{k}\rangle_{\u(N)}\langle B,X_{k}\rangle_{\u(N)}=\langle A,B\rangle_{\u(N)}.$$
The general case follows thanks to the equality $\M_{N}(\C)=\u(N)\oplus i \u(N)$ and the fact that the relations are $\C$-bilinear in $(A,B)$.

2. Choose $i,j,l,m\in \{1,\ldots,N^2\}$. By taking $A=E_{ji}$ and $B=E_{ml}$ in the first relation, we find
$$\sum_{k=1}^{N^2} (X_k)_{ij} (X_k)_{lm} = - \frac{1}{N} \delta_{i,m} \delta_{j,l}.$$
The second relation follows by developing the trace. \qed

\begin{prop}\label{ff'} Let $f:\UC \to \R$ be a differentiable function. Then $\tr f$ is differentiable and, for all $U\in \U(N)$ and all $Y\in \u(N)$, we have
\begin{equation}\label{L tr}
(\L_{Y}(\tr f))(U)=-i \tr (f'(U) Y).
\end{equation}
In particular, $\forall U\in \U(N), \; \Vert \nabla (\tr f) (U) \Vert^2 = \frac{1}{N^2} \tr(f'(U)^2)$.
\end{prop}

\begin{proof} Since $\tr f$ is invariant by conjugation, we have for all $U,V\in \U(N)$ and all $Y\in \u(N)$ the equality $(\L_{Y}(\tr f))(U)=(\L_{VYV^{-1}}(\tr f))(V U V^{-1})$. Hence, it suffices to check (\ref{L tr}) for all $Y$ when $U$ is diagonal. In this case, the result is a direct consequence of Proposition \ref{derivee f}. The second assertion follows from the definition of the gradient and the identity (\ref{coal}).
\end{proof}

\subsection{Lipschitz norms again}
At the end of the proof of Proposition \ref{main technical} (see Section \ref{est lip nor}), we will need to estimate the Lipschitz norm of a function of a unitary matrix of a special form.  We state and prove this estimation below, although the reader might want to skip it now and jump to Section \ref{con bra}.

\begin{prop}\label{lip f f} Let $f$ be an element of $C^{1,1}(\UC)$. Let $V,W$ be two elements of $\U(N)$. Define a function $F_{V,W}:\U(N)\to \C$ by setting
\[ F_{V,W}(U)=\tr \left(f'(UV) f'(UW) \right).\]
Then $F$ is Lipschitz continuous and we have the estimate
\[\| F_{V,W} \|_{\rm Lip} \leq \frac{\pi}{N} \| f'\|_{L^\infty} \| f''\|_{L^\infty}.\]
\end{prop}

\begin{proof} We prove that $F_{V,W}$ is differentiable almost everywhere on $\U(N)$ and estimate the $L^\infty$ norm of its differential. According to Proposition \ref{derivee f}, we have, for all $X\in \u(N)$ and almost all $U\in \U(N)$, the equality
\begin{align*}
(\L_X F_{V,W})(U)=&\tr\left(V^{-1} {\sf D}f'(L_{VU},R_{VU})(VUX) V f'(UW)\right) + \\
& \tr\left( f'(UV) W^{-1} {\sf D}f'(L_{WU},R_{WU})(WUX) W\right).
\end{align*}
We have used the fact that $\frac{d}{dt}_{|t=0} f'(Ue^{tX}V)=V^{-1} \frac{d}{dt}_{|t=0}f'(VUe^{tX}) V$. Let us focus on the first term of the right-hand side, the second being similar. By the Cauchy-Schwarz inequality,
\[ \left| \tr\left(V^{-1} {\sf D}f'(L_{VU},R_{VU})(VUX) V f'(UW)\right) \right|^2 \leq \tr(M^*M) \tr\left(f'(UW)^* f'(UW)\right),\]
where we have set $M={\sf D}f'(L_{VU},R_{VU})(VUX)$. 

Recall that $\M_N(\C)$ is endowed with the scalar product $\langle A,B \rangle =N \Tr(A^* B)$. We claim that the operator norm of the endomorphism ${\sf D}f'(L_{VU},R_{VU})$ of $\M_N(\C)$ with respect to this norm is bounded above by $\frac{\pi}{2}\|f''\|_{L^\infty}$. Indeed, this operator is normal with respect to this scalar product, so that  its operator norm equals its spectral radius, which is smaller than the $L^\infty$ norm of ${\sf D}f'$. Hence, we find
\[ \tr(M^*M)^{\frac{1}{2}}\leq \frac{\pi}{2}\|f''\|_{L^\infty} \tr(X^*X)^{\frac{1}{2}}.\]
It follows that
\[ \| \L_X F_{V,W} \|_{L^\infty} \leq 2 \frac{\pi}{2} \|f''\|_{L^\infty} \frac{\|X\|_{\u(N)}}{N} \|f'\|_{L^\infty},\]
from which the result follows easily. 
\end{proof}

\section{Convergence of the bracket}\label{con bra}

In this section, we prove the first two assertions of Proposition \ref{main technical}. Let us first prove a fundamental property of the generator of the Brownian motion on $\U(N)$. The action of $\U(N)$ on $\u(N)$ by conjugation is an isometric action. Hence, for all $V\in \U(N)$, the processes $U_{N}$ and $V U_{N} V^{-1}$ satisfy two stochastic differential equations (see (\ref{eds})) driven by two processes in $\u(N)$ with the same distribution, so that they have the same distribution.

\begin{lem}\label{commute} Let $F:\U(N)\to \R$ be a Lipschitz continuous function. Let $Y$ be an element of $\u(N)$. Let $t\geq 0$ be a real number. Then $\L_{Y}(P_{t}F)=P_{t}(\L_{Y}F)$.
\end{lem}

\begin{proof} Since $F$ is Lipschitz continuous, $\L_{Y}F$ is well-defined as an element of $L^\infty(\U(N))$. The result amounts simply to the interversion of an integration and a derivation: for all $U\in \U(N)$,
\begin{align*}
\L_{Y}(P_{t}F)(U)&=\frac{d}{ds}_{|s=0} \E\left[F(U e^{sY} U_{N}(t))\right]=\frac{d}{ds}_{|s=0}\E\left[F(U U_{N}(t) e^{sY})\right]\\
&=\E\left[\frac{d}{ds}_{|s=0} F(U U_{N}(t) e^{sY})\right]=P_{t}(\L_{Y}F)(U).
\end{align*}
We have used the fact that $U_{N}(t)$ has the same distribution as $e^{-sY}U_{N}(t)e^{sY}$.
\end{proof}

\subsection{It\^{o} formula} The following result summarizes the applications of It\^{o} formula that we will use. The third assertion below implies, by polarization, the first assertion of Proposition \ref{main technical}.

\begin{prop}\label{M} Let $F:\U(N)\to \R$ be an integrable function. Define a real-valued martingale $L^F$ indexed by $[0,T]$ by setting, for all $t\in [0,T]$, $L^{F}(t)=\E[F(U_{N}(T))|\F_{N,t}]$. Let $(X_k)_{k\in \{1,\ldots,N^2\}}$ be an orthonormal basis of $\u(N)$. Then the following equalities
hold for all $t\in [0,T]$.
\begin{enumerate}
\item $L^F(t)=(P_{T-t} F)(U_{N}(t))$.
\item $\displaystyle L^F(t)=L^F(0)+\int_{0}^{t} \sum_{k=1}^{N^2} \L_{X_{k}}(P_{T-s} F)(U_{N}(s))\;
d\langle X_{k},K_{N}\rangle_{\u(N)}(s)$.
\item $\displaystyle \langle L^F\rangle(t)=\int_0^t \left\Vert (\nabla (P_{T-s} F))(U_N(s)) \right\Vert^2 \; ds$.
\item If $F$ is Lipschitz continuous, then $\displaystyle \langle L^F\rangle(t)=\int_{0}^{t} \sum_{k=1}^{N^2} [P_{T-s}(\L_{X_{k}}F)(U_{N}(s))]^{2}\; ds$.
\end{enumerate}
\end{prop}

\begin{proof} 1. Choose $t\in [0,T]$. Since the unitary Brownian motion has independent multiplicative increments, $L^F(t)$ can be rewritten as
\begin{multline*}
L^F(t)=\E[F(U_{N}(T))|\F_{N,t}]=\E[F(U_{N}(t)U_{N}^*(t)U_{N}(T))|\F_{N,t}]\\=\E[F(U_{N}(t)V_{N}(T-t))|\F_{N,t}],
\end{multline*}
where $V_{N}$ is a Brownian motion on $\U(N)$ with the same distribution as $U_{N}$ and independent
of $U_{N}$. The result follows.

2. Let us apply (\ref{Ito}) to the function $G:[0,T]\times \U(N)\to \R$ defined by $G(t,U)=(P_{T-t}F)(U)$. It follows from the definition of the semigroup $(P_{t})_{t\geq 0}$ that $G$ satisfies
the time-reversed heat equation $\frac{1}{2}\Delta G+\partial_{t}G=0$. Hence, It\^{o}'s formula reads
$$L^F(t)=L^F(0)+\sum_{k=1}^{N^2}\int_{0}^{t} (\L_{X_{k}}(P_{T-s}F))(U_{N}(s)) d\langle
X_{k},K_{N}\rangle_{\u(N)}(s).$$

3. The equality follows immediately from the equality 2 and the fact that the processes  $\{\langle X_{k},K_{N}\rangle_{\u(N)}:k\in\{1,\ldots,N^2\}\}$ are independent standard real Brownian motions. 

4. This equality follows from the previous one by applying Lemma \ref{commute}.
\end{proof}

\subsection{Expectation of the bracket} \label{expbra}
We can now prove the second assertion of Proposition \ref{main technical}. Recall that we use the notation $E_{N}^{f,g}(s,U)=N^{2}\langle  \nabla (P_{T-s} (\tr f))(U) ,\nabla (P_{T-s} (\tr g))(U)\rangle_{\u(N)}$. We will use the fact, which is a consequence of Jensen's inequality, that for any square-integrable function $G:\U(N)\to \R$, and for all $t\geq 0$, $(P_{t}G)^{2}\leq P_{t}(G^2)$.

\begin{proof}[Proof of the second assertion of Proposition \ref{main technical}] Let $f:\UC \to \R$ be Lipschitz continuous.  By definition and by Lemma \ref{commute}
\begin{align*}
E_{N}^{f,f}(s,U)&=N^2 \sum_{k=1}^{N^2} (P_{T-s}(\L_{X_{k}} (\tr f)))(U)^2 \\
&\leq N^2 \sum_{k=1}^{N^2} P_{T-s}((\L_{X_{k}} \tr f)^{2})(U) 
=N^2 P_{T-s}(\Vert \nabla (\tr f)\Vert^2)(U).
\end{align*}
By Proposition \ref{ff'} and the fact that $P_{T-s}$ does not increase the uniform norm, this implies that
$$|E_{N}^{f,f}(s,U)|\leq \|f'\|_{L^\infty}^{2}.$$
By polarization, the estimation of $|E^{f,g}_{N}(s,U)|$ follows.

Now, let us consider two
independent copies $V_{N}$ and $W_{N}$ of the unitary Brownian motion $U_{N}$. Then, denoting by $\E_{V_{N},W_{N}}$ the expectation with respect to $V_{N}$ and $W_{N}$ only, we have
\begin{align*}
&E_{N}^{f,f}(s,U_{N}(s))=N^2 \sum_{k=1}^{N^2} (P_{T-s}(\L_{X_{k}} (\tr f)))(U_N(s))^2\\
&=N^2 \sum_{k=1}^{N^2} \E_{V_{N},W_{N}}[(\L_{X_{k}} \tr f)(U_{N}(s)V_{N}(T-s))(\L_{X_{k}} \tr f)(U_{N}(s)W_{N}(T-s))].
\end{align*}
Using successively Proposition \ref{ff'} and Lemma \ref{trace}, we find
$$E_{N}^{f,f}(s,U_{N}(s))=\E_{V_{N},W_{N}}\left[\tr\left(f'(U_{N}(s)V_{N}(T-s)) f'(U_{N}(s)W_{N}(T-s))\right)\right].$$

Taking the expectation with respect to $U_{N}$, we find finally
$$\E[E_{N}^{f,f}(s,U_{N}(s))]=\E\left[\tr\left(f'(U_{N}(s)V_{N}(T-s)) f'(U_{N}(s)W_{N}(T-s))\right)\right].$$

Let $({\mathcal A},\tau)$ be a $C^{*}$-probability space which carries three free mutliplicative brownian motions $u,v,w$ which are mutually free. According to Theorem \ref{cv mb mbul}, the family $(U_{N}(s),V_{N}(t),W_{N}(u))_{s,t,u\geq 0}$, seen as 
a collection of non-commutative random variables in the non-commutative probability space $(L^{\infty}\otimes \M_{N}(\C),\E\otimes \tr)$, converges in distribution to $(u_{s},v_{t},w_{u})_{s,t,u\geq 0}$ as $N$ tends to infinity. This implies in particular that for all non-commutative polynomial $p$ in three variables and their adjoints, and for all $s,t,u\geq 0$ 
$$\E[\tr \; p(U_{N}(s),V_{N}(t),W_{N}(u))] \build{\longrightarrow}_{N\to \infty}^{} \tau(p(u_{s},v_{t},w_{u})).$$

Let us fix $s\in [0,T)$. Since $\mathcal A$ is a $C^*$-algebra, there is a continuous functional calculus on normal elements, hence on unitary elements, and $f'(u_{s}v_{T-s})f'(u_{s}w_{T-s})$ is a well-defined element of $\mathcal A$. On the other hand, choose $\epsilon>0$ and let $q(z,w)$ be a polynomial function in $z,w$ and their adjoints such that $\sup_{z,w\in \UC}|f'(z)f'(w)-q(z,w)|<\epsilon$.
Then
\begin{align*}
&\left| \E\left[\tr\left(f'(U_{N}(s)V_{N}(T-s)) f'(U_{N}(s)W_{N}(T-s))\right)\right] - \tau(f'(u_{s}v_{T-s})f'(u_{s}w_{T-s}))\right| \\
&\hspace{0.5cm}\leq \left| \E\left[\tr\left(f'(U_{N}(s)V_{N}(T-s)) f'(U_{N}(s)W_{N}(T-s))\right)\right.\right.\\
& \hspace{5cm} \left.\left.- \tr\; q(U_{N}(s)V_{N}(T-s),U_{N}(s)W_{N}(T-s))\right] \right| \\
&\hspace{0.5cm}+ \left| \E\left[\tr\; q(U_{N}(s)V_{N}(T-s),U_{N}(s)W_{N}(T-s))\right] - \tau(q(u_{s}v_{T-s},u_{s}w_{T-s}))\right|\\
& \hspace{0.5cm}+ \left| \tau(q(u_{s}v_{T-s},u_{s}w_{T-s})) - \tau(f'(u_{s}v_{T-s})f'(u_{s}w_{T-s}))\right|.
\end{align*}

The first and the third term are smaller than the uniform distance between $q(\cdot,\cdot)$ and $f'(\cdot)f'(\cdot)$, hence smaller than $\epsilon$. The middle term tends to $0$ as $N$ tends to infinity. Altogether, this proves that
$$\E[E_{N}^{f,f}(s,U_{N}(s))] \build{\longrightarrow}_{N\to\infty}^{} \tau(f'(u_{s}v_{T-s})f'(u_{s}w_{T-s})),$$
from which the expected result follows by polarization.
\end{proof}

\section{Convergence of the variance of the bracket} \label{con var}
This section is devoted to the proof of the third assertion of Proposition \ref{main technical}.

\subsection{A weak concentration inequality} 

Consider a function $F:\U(N)\to \R$. If $F$ is Lipschitz continuous, then  the equality $\Vert F \Vert_{\rm Lip} = \Vert \nabla F \Vert_{L^{\infty}}$ holds. The goal of this paragraph is to prove the following inequality.

\begin{prop}\label{concentre} Let $F:\U(N)\to \R$ be a Lipschitz continuous function. For all
$T\geq 0$, one has the following inequality:
$$\Var[F(U_N(T))]\leq T \Vert F \Vert_{\rm Lip}^2.$$
\end{prop}

Note that this inequality is preserved by rescaling of the Riemannian metric on $\U(N)$, that is, by rescaling of the scalar product on $\u(N)$. Indeed, let $\lambda$ be a positive real and
let us consider the scalar product $\langle \cdot, \cdot \rangle_{\u}^{\widetilde{}} =\lambda
\langle \cdot, \cdot\rangle_{\u}$ on $\u(N)$. Then, putting a tilda to the quantities associated
with this new scalar product, we have on one hand $\tilde d=\lambda^{\frac{1}{2}} d$ and $\Vert F
\Vert_{\widetilde{\rm Lip}} =\lambda^{- \frac{1}{2}}\Vert F \Vert_{\rm Lip}$, and on the other
hand $\widetilde \Delta = \lambda^{-1} \Delta$ and $\widetilde U_N(T)$ has the distribution of $U_N({\lambda}^{-1}{T})$.

\begin{proof} Recall the definition of the martingale $L^F$ (see Proposition \ref{M}).  The
left-hand side is equal to $\E[\langle L^F \rangle (T)]$, thus, by the third assertion of Proposition
\ref{M}, to 
\begin{multline*}
\E \int_{0}^{T} \Vert (\nabla (P_{T-s}F))(U_{N}(s))\Vert^2\; ds \leq T \sup_{s\in [0,T)} \Vert \nabla (P_{T-s} F)\Vert^{2}_{L^{\infty}}\\=T \sup_{s\in [0,T)} \Vert P_{T-s}F \Vert^{2}_{\rm Lip}.
\end{multline*}
On the other hand, since $F$ is Lipschitz continuous, for all $t\geq0$, $\| P_{t} F \|_{\rm Lip}\leq \|F \|_{\rm Lip}$. The result follows.
\end{proof}

\subsection{An estimate of a Lipschitz norm}\label{est lip nor}

With Proposition \ref{concentre} in mind, we  are going to study the Lipschitz norm of $U\mapsto E_N^{f,f}(s,U)$ in order to estimate the variance of $E_{N}^{f,f}(s,U_{N}(s))$. 

\begin{prop}\label{conc s} Assume that $f$ is of class $C^{1,1}(\UC)$. Then 
$$\sup_{s\in [0,T]} \Vert E_N^{f,f}(s,\cdot) \Vert_{\rm Lip}=O(N^{-1}).$$
\end{prop}

\begin{proof} The proof relies on the identity
\[E_N^{f,f}(s,U_N(s))=\E_{V_N,W_N} \left[ \tr\left( f'(U_N(s) V_N(T-s)) f'(U_N(s) W_N(T-s)) \right) \right].\]
By Proposition \ref{lip f f}, the expression between the brackets is a Lipschitz continuous function  of $U_N(s)$ for all values of $V_N(T-s)$ and $W_N(T-s)$, with a Lipschitz norm which does not depend on $V_N(T-s)$ and $W_N(T-s)$ and is $O(N^{-1})$. Hence, the same estimate holds for the expectation.
\end{proof}

\begin{proof}[Proof of the third assertion of Proposition \ref{main technical}]  It suffices to combine Proposition \ref{conc s} and Proposition
\ref{concentre} to find that that $\sup_{s\in [0,T)}\Var[E_N^{f,f}(s,U_N(s))]=O(N^{-2})$.
The same result for $E_{N}^{f,g}$ follows easily.
\end{proof}

This concludes the proof of Proposition \ref{main technical} and thus of Theorem \ref{main theo}.

\section{Other Brownian motions, on unitary and special unitary groups}\label{sec su}
\label{u su}

In this section, we explain how Theorem \ref{main theo} can be extended to other Brownian motions on the unitary group and to the Brownian motion on the special unitary group.

In this paper so far, we have considered the Brownian motion $U_N$ on $\U(N)$
associated to the scalar product on $\u(N)$ given by 
$\langle X,Y \rangle_{\u(N)}= N \Tr(X^{*}Y),$ for any $X,Y \in \u(N)$. The crucial property of this scalar product is its invariance under the action of $\U(N)$ on $\u(N)$ by conjugation. There is in fact a two-parameter family of scalar products with this invariance property, namely $a\Tr\left((X-\tr X)^*(Y-\tr Y)\right) + b \Tr(X^*) \Tr(Y)$ with $a,b>0$. Multiplying the two parameters $a$ and $b$ by the same constant simply affects the Brownian motion by a global rescaling of time, indeed dividing time by this constant, so that we may choose the value of one of them. We take $a=N$ in order to have correct asymptotics as $N$ tends to infinity. This choice being made, varying $b$ really yields different Brownian motions. It turns out to be more convenient to take $\alpha=b^{-\frac{1}{2}}$ as the parameter: we define, for all $\alpha>0$,  the scalar product
\[\langle X,Y \rangle_{\u(N)}^{(\alpha)}= N \Tr((X-\tr X)^{*}(Y-\tr Y))+\frac{1}{\alpha^{2}}\Tr(X^{*})\Tr(Y)\]
on $\u(N)$. In particular, the scalar product considered in the rest of this paper corresponds to $\alpha=1$.

In order to understand the Brownian motions associated to the scalar products $\langle \cdot, \cdot \rangle_{\u(N)}^{(\alpha)}$, we start by defining the Brownian motion on $SU(N)$, which corresponds to the limit where $\alpha$ tends to $0$.

Let us denote by $\su(N)$ the hyperplane of $\u(N)$ consisting of traceless matrices, which is also the Lie algebra of the special unitary group $\SU(N)$, and let $K_{N}^{0}$ be the linear Brownian motion on $\su(N)$ corresponding to the scalar product induced by $\langle \cdot, \cdot \rangle_{\u(N)}$. Let $V_N$ be the solution of the stochastic differential equation
\begin{equation}\label{edssu}
dV_{N}(t)=V_{N}(t)dK_{N}^{0}(t)-\frac{1}{2}\left(1-\frac{1}{N^{2}}\right)V_{N}(t)dt.
\end{equation}
One can check that if the initial condition is in  the special unitary group, then the process $V_N$ stays in it: the constant $1-\frac{1}{N^{2}}$ is designed for that purpose. We call $V_{N}$ the Brownian motion on $\SU(N)$. 


Now, for all $\alpha \geq 0$, let us consider the following process with values in $\U(N)$:
\[V_{N}^{(\alpha)} (t)=e^{\frac{i \alpha B_{t}}{N}} V_{N}(t),\]
where $(B_{t})_{t\geq 0}$ is a standard real Brownian motion independent of $V_{N}.$ 
Let $(Y_{1},\ldots,Y_{N^{2}-1})$ be an orthonormal basis of $\su(N)$. 
For all $\alpha\geq 0$, the generator of $V_N^{(\alpha)}$ is given by
\[\frac{1}{2}\Delta^{(\alpha)}=\frac{1}{2}\left( \sum_{k=1}^{N^{2}-1}\L_{Y_{i}}^{2}+\alpha^{2}\L_{\frac{i}{N}I_{N}}^{2}\right),\]
and we call $V_{N}^{(\alpha)}$ the $\alpha$-Brownian motion on $\U(N)$.

For each $\alpha> 0$, the process $V_{N}^{(\alpha)}$ is naturally associated with the scalar product $\langle X,Y \rangle_{\u(N)}^{(\alpha)}$ on $\u(N)$. Indeed, let $K_N^{(\alpha)}$ be the linear Brownian motion on $\u(N)$ corresponding to this scalar product. It can be expressed as 
$K_N^{(\alpha)}=K_{N}^{0}+\frac{i\alpha}{N}B.$ Then the process $V_{N}^{(\alpha)}$
satisfies the stochastic differential equation
\begin{equation}\label{edsalpha}
dV_{N}^{(\alpha)}(t)=V_{N}^{(\alpha)}(t)dK_{N}^{(\alpha)}(t)-\frac{1}{2}\left(1+\frac{\alpha^{2}-1}{N^{2}}\right)V_{N}^{(\alpha)}(t)dt.
\end{equation}
In particular, $V_{N}^{(1)}$ has the same distribution as $U_{N}$.

The main feature of the Brownian motion on $\U(N)$ which we have used extensively in the proof of Theorem \ref{main theo} is that its generator commutes with all Lie derivatives. Since the Lie derivative in the direction of $i I_{N}$ commutes with all Lie derivatives, this is also the case for the generator of $V_N$ and of all the processes $V_{N}^{(\alpha)}$, $\alpha>0$.

Finally, following \cite{Biane}, one can check that for all $\alpha\geq 0$, the process $V_{N}^{(\alpha)}$ converges as $N$ tends to infinity to a free multiplicative Brownian motion.

Let us now define a modified version of the covariance $\sigma_T$.

\begin{defi}\label{sigmaalpha} With all the notation of Definition \ref{sigma}, we define, for all $\alpha\geq 0$,
$$\sigma_{T}^{(\alpha)}(f,g)= \int_{0}^{T} \tau(f'(u_{s}v_{T-s})g'(u_{s}w_{T-s}))+(\alpha^{2}-1)\tau(f'(u_{s}v_{T-s}))\tau(g'(u_{s}w_{T-s}))\; ds.$$
\end{defi}

Following step by step the proof of Theorem \ref{main theo}, one finds the following result.

\begin{theo}\label{main theo su} Let $T\geq 0$ be a real number. Let $n\geq 1$ be an integer. Let $f_{1},\ldots,f_{n}:\UC\to \R$ be $n$ functions of $C^{1,1}(\UC)$. Let us define a $n\times n$ real non-negative symmetric matrix by setting  $\Sigma_T^{(\alpha)}(f_1,\ldots,f_n)=(\sigma_{T}^{(\alpha)}(f_{i},f_{j}))_{i,j\in\{1,\ldots,n\}}$. Then, as $N$ tends to infinity, the following convergence of random vectors in $\R^{n}$ holds in distribution:
\begin{equation}
N\left(\tr f_{i}(V_{N}^{(\alpha)}(T)) - \E\left[\tr f_{i}(V_{N}^{(\alpha)}(T))\right]\right)_{i\in\{1,\ldots,n\}} \build{\longrightarrow}_{N\to\infty}^{(d)} {\mathcal N}(0, \Sigma_T^{(\alpha)}(f_1,\ldots,f_n)).
\end{equation}
\end{theo}

We leave the details to the reader, since every step can be adapted in a straightforward way. The only substantial change is in Lemma  \ref{trace}, which now will take the following form.

\begin{lem}\label{tracealpha} Let $(X_k)_{k\in \{1,\ldots,N^2\}}$ be an orthonormal basis of $(\u(N),\langle \cdot,\cdot\rangle_{\u(N)}^{(\alpha)})$. Let
$A,B$ be elements of $\M_N(\C)$. Then the following equality holds:
\begin{equation}
\label{coaltrace}
\sum_{k=1}^{N^2} \tr(AX_k) \tr(BX_k) =-\frac{1}{N^2} \left(\tr(AB)+(\alpha^{2}-1)\tr(A)\tr(B)\right).
\end{equation}
Assume that $(X_{1},\ldots,X_{N^{2}-1})$ form an orthonormal basis of $\su(N)$ endowed with the scalar product induced by $\langle \cdot,\cdot\rangle_{\u(N)}$. Then
\begin{equation}
\label{fragtrace}
\sum_{k=1}^{N^2-1} \tr(AX_k)\tr(BX_k) =-\frac{1}{N^{2}}\left(\tr(AB)-\tr(A) \tr(B)\right).
\end{equation}
\end{lem}

It is this modification which gives rise to the new covariance introduced in Definition \ref{sigmaalpha}.\\

\section{Joint fluctuations of the unitary Brownian motion at different times}
\label{sec multitime}

A natural generalization of our main result consists in considering several Brownian motions stopped at possibly 
different times.
The goal of this section is to establish an analogue of Theorem \ref{main theo} in this case.  In order to state the result, we define a new covariance function.

\begin{defi}\label{sigma multi} Let $(\A,\tau)$ be a $C^*$-probability space which carries three free multiplicative Brownian motions $u,v,w$ which are mutually free. Let $T_1,T_2\geq 0$ be real numbers. Let $f,g:\UC\to \R$ be two functions of $C^1(\UC)$. For all $s\in [0,T_1\wedge T_2]$, we set
$\sigma_{T_1,T_2,s}(f,g)=\tau(f'(u_{s}v_{T_1-s})g'(u_{s}w_{T_2-s}))$. Then, we define
$$\sigma_{T_1,T_2}(f,g)= \int_{0}^{T_1\wedge T_2} \sigma_{T_1,T_2,s}(f,g)\; ds =\int_{0}^{T_1\wedge T_2}  \tau(f'(u_{s}v_{T_1-s})g'(u_{s}w_{T_2-s}))\; ds.$$
\end{defi}

We have the following result.

\begin{theo}\label{theo multi}  Let $n\geq 1$ be an integer.  Let $T_1,\ldots,T_n\geq 0$ be real numbers. Let $f_{1},\ldots,f_{n}:\UC\to \R$ be $n$ functions of $C^{1,1}(\UC)$. Let us define a $n\times n$ real non-negative symmetric matrix by setting  $\Sigma_{T_1, \ldots,T_n}(f_1,\ldots,f_n)=(\sigma_{T_i,T_j}(f_{i},f_{j}))_{i,j\in\{1,\ldots,n\}}$. Then, as $N$ tends to infinity, the following convergence of random vectors in $\R^{n}$ holds in distribution:
\begin{equation}\label{main multi}
N\left(\tr f_{i}(U_{N}(T_i)) - \E\left[\tr f_{i}(U_{N}(T_i))\right]\right)_{i\in\{1,\ldots,n\}} \build{\longrightarrow}_{N\to\infty}^{(d)} {\mathcal N}(0, \Sigma_{T_1, \ldots,T_n}(f_1,\ldots,f_n)).
\end{equation}
\end{theo}

The proof of this result is very similar to the proof of Theorem \ref{main theo} and, as in the previous section, we simply point out the small differences between the two.

For the sake of convenience, let us assume  $T_1 \le \ldots \le T_n$. Let $f_{1},\ldots,f_{n}:\UC\to \R$ be $n$ functions of $C^{1,1}(\UC)$. We define for each $i\in\{1,\ldots,n\}$ a martingale indexed by $[0,T_n]$ by setting
$$ M_N^{f_i}(t) = \E(\tr f_i(U_N(T_i)) | \mathcal F_{N,t}).$$
Observe that the martingale $M_N^{f_i}$ is constant on the interval $[T_i,T_n]$.
Let us now define the vector-valued martingale $Q_N(t) = N\left( M_N^{f_i}(t)-M_N^{f_i}(0) \right)_{i \in \{1, \ldots,n\}}$, so that the left hand-side of \eqref{main multi} is equal to $Q_N(T_n)$. The proof of Theorem \ref{theo multi} relies on an analogue of Proposition \ref{main technical}, for which we introduce the following notation : for all $i,j\in\{1,\ldots,n\}$ with $i\leq j$ and all $s\in [0,T_i)$, we set

\[E_{N}^{f_i,f_j}(s,U)=N^2 \langle \nabla (P_{T_i-s}(\tr f_i))(U),\nabla (P_{T_j-s}(\tr f_j))(U) \rangle_{\u(N)}.\]
We state the following result for the two functions $f_1$ and $f_2$.

\begin{prop} \label{multi technical} With the notation introduced above, the following properties hold.
\begin{enumerate}
\item For all $t\in [0,T_2]$, the quadratic covariation of the martingales $NM^{f_1}_{N}$ and $NM^{f_2}_{N}$ is given by
$$\langle N M^{f_1}_{N}, N M^{f_2}_{N}\rangle_{t}=\int_{0}^{t \wedge T_1} E_{N}^{f_1,f_2}(s,U_{N}(s))\; ds.$$
\item Assume that $f_1$ and $f_2$ are Lipschitz continuous. Then for all $s\in [0,T_1)$ and all $U\in \U(N)$, 
$| E_{N}^{f_1,f_2}(s,U)|\leq \left(\Vert f_1 \Vert_{\rm Lip}+\Vert f_2 \Vert_{\rm Lip}\right)^2$. Moreover, if $f_1$ and $f_2$ belong to $C^{1}(\UC)$, then the following convergence holds:
$$\E[E_{N}^{f_1,f_2}(s,U_{N}(s))]\build{\longrightarrow}_{N\to\infty}^{} \sigma_{T_1,T_2,s}(f_1,f_2).$$
\item Assume that $f_1$ and $f_2$ belong to $C^{1,1}(\UC)$. Then the following estimate holds:
$$\sup_{s\in [0,T_1)} \Var(E_{N}^{f_1,f_2}(s,U_{N}(s)))=O(N^{-2}).$$
\end{enumerate}
\end{prop}

The proof of this Proposition is in no way different from the proof of Proposition \ref{main technical}. The unique novelty is the fact that $M_N^{f_1}$ is constant on the interval $[T_1,T_2]$ so that the quadratic covariation $\langle M_N^{f_1},M_N^{f_2}\rangle$ vanishes on this interval.

Then, one deduces Theorem \ref{theo multi} from Proposition \ref{multi technical} just as one deduces Theorem \ref{main theo} from Proposition \ref{main technical}.

Let us mention that, in the case where the functions $f_1,\ldots,f_n$ are polynomial, and given Theorem \ref{main theo}, the Gaussian character of the fluctuations in the case where the Brownian motions are stopped at different times is a consequence of the work of J. Mingo, R. Speicher and P. \'Sniady \cite{sof1,sof2} on the notion of second order freeness and its specialization to the case of unitary matrices. Their work also provides one with a covariance function and it could be interesting to investigate the relation between our expression of what we call $\sigma_{T_1,T_2}$ and theirs.

Another natural question which is answered by the theory of second order freeness is that of the asymptotic fluctuations of random variables of the form $\tr p(U_N(T_1),\ldots,U_N(T_k))$ where $p$ is a non-commutative polynomial. It seems more difficult, although not hopeless, to apply our techniques to such functionals. 

\section{Behaviour of the covariance for large time}
\label{diaco}

For any fixed $N$, the Markov process $(U_N(T))_{T \geq 0}$ converges in distribution, as $T$ goes to infinity, to its invariant measure, which is the Haar measure on $\U(N).$ In \cite{DE}, P.~Diaconis and S.~Evans established a central limit theorem for Haar distributed unitary random matrices. In this section, we relate our result to theirs by comparing the limit as $T$ tends to infinity of the covariance $\sigma_{T}$ with the
covariance which they have found.

\subsection{Statement of the result of convergence}
In order to state the result of Diaconis and Evans, we need to introduce some notation.

\begin{defi} \label{H1/2}
Let $\H^\frac{1}{2}(\UC)$ denote the space of functions that are square-integrable on $\UC$ and such that
$$ \Vert f\Vert^2_{\frac{1}{2}} := \frac{1}{16\pi^2}\int_{[0,2\pi]^2} \frac{\left|f(e^{i\varphi})-f(e^{i\theta})\right|^2}{\sin^2\left(\frac{\varphi-\theta}{2}\right)}\; d\phi d\theta <\infty.$$
We denote by $\langle\cdot, \cdot\rangle_{\frac{1}{2}}$ the inner product associated to this Hilbertian semi-norm.
\end{defi}

For all $f: \UC \rightarrow \C$ which is square-integrable and all $j \in \Z$, we denote by $a_j(f)= \frac{1}{2\pi} \int_\UC f(\xi) e^{-ij\xi}d\xi$ the $j$-th Fourier coefficient of $f$. One can check that $f \in H^{\frac{1}{2}}(\UC)$ if and only if $\sum_{j \in \Z} |j| |a_j(f)|^2$ is finite and that, in this case,
\[\Vert f\Vert^2_{\frac{1}{2}}=\sum_{j \in \Z} |j| |a_j(f)|^2.\]


The result of Diaconis and Evans states as follows.

\begin{theo}(5.1 in \cite{DE}) \label{tcl-haar}
For all $N \in \N,$ let $M_N$ be a $N \times N$ unitary matrix distributed according to the Haar measure 
on $\U(N).$ Let $n  \geq 1$ be an integer. For all $f_1, \ldots, f_n \in H^{\frac{1}{2}}(\UC)$, let  $\Sigma(f_1, \ldots, f_n) $ be the $n \times n $ real non-negative symmetric matrix defined
by $\Sigma(f_1, \ldots, f_n) = \left(\langle f_i,f_j\rangle_{\frac{1}{2}}\right)_{i,j= 1,\ldots,n}.$ 
As $N$ goes to infinity, the following convergence of random vectors in $\R^n$ holds in distribution:
$$
N\left(\tr f_{i}(M_N) - \E\left[\tr f_{i}(M_N)\right]\right)_{i\in\{1,\ldots,n\}} \build{\longrightarrow}_{N\to\infty}^{(d)} {\mathcal N}(0, \Sigma(f_1, \ldots, f_n)).
$$
\end{theo}

In view of this result, it is natural to expect the covariance that we have introduced in Definition \ref{sigma} to converge, as $T$ tends to infinity, to the covariance given by the $H^{\frac{1}{2}}$-scalar product. This is what the following result expresses.

\begin{theo} \label{cvtohaar}
For all $n \geq 1$ and all $f_1, \ldots, f_n \in  H^{\frac{1}{2}}(\UC),$ 
$$ \Sigma_T(f_1, \ldots, f_n)\build{\longrightarrow}_{T\to\infty}^{} \Sigma(f_1, \ldots, f_n).$$
\end{theo}

Let us emphasize that $\Sigma_T(f_1, \ldots, f_n)$ has only been defined so far for functions in $C^{1,1}(\UC)$. From this point on, we focus on extending the definition of the covariance to functions of the space $H^{\frac{1}{2}}(\UC)$ and proving  Theorem \ref{cvtohaar}.

\subsection{The main estimate} 
In the sequel, $(u_t)_{t \geq 0}, (v_t)_{t \geq 0}$ and $(w_t)_{t \geq 0}$ will be three multiplicative free Brownian motions, that are mutually free. 
For all $T \geq 0$ and all $k \in \Z,$ let us denote by $\mu_k(T)=\tau(u_T^k)$ the $k$-th moment of $u_T$. Recall that, since $u_T$ has the same law as $u_T^*$, one has, for all $k\in \Z$, the equality $\mu_k(T)=\mu_{-k}(T)$. For each $k\geq 1$, according to \cite{Biane}, $\mu_{k}(T)$ is given by
\begin{equation}\label{moments}
\mu_{k}(T)=e^{-\frac{kT}{2}}\sum_{l=0}^{k-1}\frac{(-T)^{l}}{l!} \binom{k}{l+1} k^{l-1}.
\end{equation}

\begin{lem} \label{bdmom}
For all $\varepsilon>0$, all $T\geq T_0(\varepsilon)=\frac{2}{\epsilon} \log(1+\frac{2}{\epsilon})$ and all $k \in \Z$, one has
$$ \vert \mu_k(T)\vert \leq e^{-|k|T\left(\frac{1}{2}- \varepsilon\right)}.$$
\end{lem}

\begin{proof}
If $k=0$ or $\epsilon\geq \frac{1}{2}$, the inequality is trivial. Moreover, since $\mu_k(T)=\mu_{-k}(T)$, it suffices to prove the inequality for $k> 0$. So, let us assume that $\epsilon\leq \frac{1}{2}$ and $k \in \N^*$. It is easy to check that the expression \eqref{moments} of $\mu_{k}(T)$ is equivalent to the following:
$$ \mu_k(T) = \frac{e^{-\frac{kT}{2}}}{2ik\pi} \oint e^{-kTz} \left(1+ \frac{1}{z} \right)^k dz,$$
where we integrate over a closed path of index 1 around the origin of the complex plane.
If we choose as our contour the circle of radius $\frac{\epsilon}{2}$ centered at the origin,
we get 
$$ \mu_k(T) = \frac{e^{-\frac{kT}{2}}}{2ik\pi} \int_0^{2\pi} e^{-kT\frac{\varepsilon}{2} e^{i\t}} \left(1+ \frac{2}{\varepsilon e^{i\t}} \right)^k i\frac{\varepsilon}{2} e^{i\t}d\t,$$
so that, provided $T\geq T_0(\varepsilon),$
$$ \vert \mu_k(T)\vert \leq \frac{\epsilon}{2k}e^{-\frac{kT}{2}}  e^{kT\frac{\varepsilon}{2}} \left(1+ \frac{2}{\varepsilon} \right)^k  \leq e^{-kT\left(\frac{1}{2}- \varepsilon\right) },$$
as expected.
\end{proof}

We will denote by $T_0$ a real large enough such that for all
$T \geq T_0$ and all $k \in \Z$, the inequality  $\vert \mu_k(T)\vert \leq e^{-|k|\frac{T}{3}}$ holds. One can check that $31$ is large enough but we choose $T_0=32$ for reasons which will soon become apparent.

For all $j,k \in \Z$ and  $T >0,$ we define
\begin{equation}\label{deftaujk}
\tau_{j,k}(T) = \int_0^T \tau\left((u_s v_{T-s})^j(u_s w_{T-s})^k\right)ds.
\end{equation}

\begin{prop}\label{taujk} Set $T_0=32$. For all $T\geq T_0$ and all $(j,k)\neq (0,0)$, the following inequality holds:
\begin{equation}\label{main jk}
|\tau_{j,k}(T)| \leq 4 \frac{e^{-\frac{|j+k|}{4}T}}{|j|+|k|} + (|j|+|k|) T_0 e^{-\frac{|j|+|k|}{4}(T-T_0)}.
\end{equation}
Moreover, if $j\neq 0$, then
\begin{equation}\label{main j}
\left|\tau_{j,-j}(T)- \frac{1}{|j|}\right| \leq \frac{e^{-\frac{T}{4}}}{|j|}+ 2 |j| T_0 e^{-\frac{|j|}{2}(T-T_0)}.
\end{equation}
In particular, for all $(j,k)\neq (0,0)$, the following convergence holds :
\[\lim_{T \rightarrow \infty} \tau_{j,k}(T) = \delta_{j+k,0} \frac{1}{|j|}.\]
\end{prop}

The proof of these estimates relies on a differential system satisfied by the functions $\tau_{j,k}$. This differential system is a consequence of the free It\^o calculus for free multiplicative Brownian motions. We state the form that we use, which is of interest on its own.

\begin{prop}\label{freeitobm}
Let $(u_t)_{t\geq 0}$ be a free multiplicative Brownian motion on some non-commutative $*$-probability space $(\mathcal A,\tau)$. Let $a_1,\ldots,a_n \in \mathcal A$ be random variables such that the two families $\{u_t : t\geq 0\}$ and $\{a_1,\ldots,a_n\}$ are free. Finally, choose $\epsilon_1,\ldots,\epsilon_n\in \{1,*\}$. Then
\begin{align*}
\frac{d}{dt} \tau(u_t^{\epsilon_1} a_1 \ldots u_t^{\epsilon_n} a_n) =& -\frac{n}{2} \tau(u_t^{\epsilon_1} a_1 \ldots u_t^{\epsilon_n} a_n) \\
&- \sum_{1\leq i<j \leq n} {\mathbbm{1}}_{\epsilon_i=\epsilon_j} \tau(a_i \ldots a_{j-1} u_t^{\epsilon_j}) \tau(a_j \ldots a_{i-1} u_t^{\epsilon_i}) \\
&+ \sum_{1\leq i<j \leq n} {\mathbbm{1}}_{\epsilon_i\neq \epsilon_j} \tau(a_i \ldots a_{j-1}) \tau(a_j \ldots a_{i-1}),
\end{align*}
where for all $1\leq i<j\leq n$, we have used the shorthands $a_i\ldots a_{j-1}$ for  $a_i u_t^{\epsilon_{i+1}} a_{i+1} \ldots u_t^{\epsilon_{j-1}} a_{j-1}$ and $a_j\ldots a_{i-1}$ for $a_j u_t^{\epsilon_{j+1}}a_{j+1} \ldots u_t^{\epsilon_n} a_nu_t^{\epsilon_1} a_1 \ldots u_t^{\epsilon_{i-1}} a_{i-1}$.
\end{prop}

\begin{proof} In \cite{Biane}, P.~Biane showed that the free multiplicative Brownian motion
$(u_t)_{t\geq 0}$ satisfies the free stochastic differential equation $du_t=i u_t  dx_t -\frac{1}{2} u_t dt,$ where $(x_t)_{t\geq 0}$ is a free additive (Hermitian) Brownian motion. 
The identity above follows from this fact by free stochastic calculus, which  has been developed by P. Biane and R. Speicher and is exposed in \cite{BianeSpeicher}. For the reader not familiar with free stochastic calculus, and without entering into the details, let us explain how the computation goes. The analogy with usual It\^{o} calculus should be a helpful guide. 

The equation satisfied by $u_{t}$ implies that $u_{t}^{*}$ satisfies the equation $du_{t}^{*}=-idx_{t}u_{t}^{*}-\frac{1}{2}u_{t}^{*}dt$. The time derivative of $\tau(u_t^{\epsilon_1} a_1 \ldots u_t^{\epsilon_n} a_n)$ is computed formally by applying the formula
\begin{multline*}
d \left(\tau(u_t^{\epsilon_1} a_1 \ldots u_t^{\epsilon_n} a_n)\right) =\sum_{i=1}^{n} \tau(u_t^{\epsilon_1} a_1 \ldots du_{t}^{\epsilon_{i}} \ldots u_t^{\epsilon_n} a_n)\\
+\sum_{1\leq i< j \leq n} \tau(u_t^{\epsilon_1} a_1 \ldots du_{t}^{\epsilon_{i}} \ldots du_{t}^{\epsilon_{j}} \ldots u_t^{\epsilon_n} a_n),
\end{multline*}
together with the rules

\[\tau(a  \, dt)=\tau(a)dt \; , \;\; \tau(a \,dx_{t})=0 \; , \;\; \tau (a\, dt \,b \,dt)=\tau (a\, dt\, b \,dx_{t})=0 \; , \]
\[ \textrm{and } \tau(a\, dx_{t}\, b \, dx_{t})=\tau(a)\tau(b)dt\]
valid for all $a,b\in \A$, and using the invariance of $\tau$ under cyclic permutation of its arguments.
\end{proof}

\begin{lem}\label{freeito}
The family $(\tau_{j,k})_{(j,k)\in \Z^2}$
satisfies the following system of differential equations : 
\begin{align*}
\dot\tau_{j,k}(T) =  \mu_{j+k}(T) - \frac{|j|+|k|}{2} \tau_{j,k}(T) 
& - \sum_{l=1}^{|j|-1} (|j|-l) \mu_l(T) \tau_{{\rm sgn}(j)(|j|-l),k}(T) \\
& - \sum_{m=1}^{|k|-1} (|k|-m) \mu_m(T) \tau_{j,{\rm sgn}(k)(|k|-m)}(T),
\end{align*}
where $\dot\tau_{j,k}$ is the derivative of the function $T \mapsto \tau_{j,k}(T).$
\end{lem}

\begin{proof} This differential system follows easily from an application of Proposition \ref{freeitobm} to the expression \eqref{deftaujk}. 
\end{proof}

Before we turn to the proof of Proposition \ref{taujk}, let us state some elementary properties of the functions $\tau_{j,k}$. For all $k\geq 0$, define the polynomial $P_{k}$ by the relation $\mu_{k}(T)=e^{-\frac{kT}{2}}P_{k}(T)$. For $k< 0$, define $P_{k}=P_{-k}$.

\begin{lem} For all $j,k\in \Z$, the function $\tau_{j,k}$ is real-valued and satisfies $\tau_{j,k}=\tau_{k,j}=\tau_{-j,-k}$. Moreover, there exists a family of polynomials $(R_{j,k})_{j,k\in \Z}$ with rational coefficients such that the following equality holds :
\begin{equation}\label{defR}
\forall j,k\in \Z \; , \;\; \tau_{j,k}(T)=\frac{\mathbbm 1_{j\neq 0}}{|j|} \delta_{j+k,0} + e^{- \frac{|j|+|k|}{2} T} R_{j,k}(T).
\end{equation}
These polynomials are characterized by the fact that for all $j,k\in \Z$, $R_{j,k}(0)=0$ and 
\begin{equation}\label{sysR}
\dot R_{j,k}={\mathbbm 1}_{jk\geq 0} P_{j+k}-\sum_{l=1}^{|j|-1} (|j|-l) P_{l}R_{{\rm sgn}(j)(|j|-l),k}- \sum_{m=1}^{|k|-1} (|k|-m) P_{m} R_{j,{\rm sgn}(k)(|k|-m)}.
\end{equation}
\end{lem}

\begin{proof} The equalities $\tau_{j,k}=\overline{\tau_{-j,-k}}=\tau_{k,j}$ follow from the definition of $\tau_{j,k}$, using the unitarity of $u,v,w$, the traciality of $\tau$, and the fact that the families $(u,v,w)$ and $(u,w,v)$ have the same joint distribution. The fact that $\tau_{j,k}$ is real-valued can be proved by induction using the differential system stated in Lemma \ref{freeito}, or directly using the definition and the fact that $(u,v,w)$ and $(u^{*},v^*,w^*)$ have the same distribution. 

The functions $R_{j,k}$ defined by \eqref{defR} are easily checked to satisfy the differential system \eqref{sysR} and, by induction, to be polynomial.
\end{proof}

\begin{proof}[Proof of Proposition \ref{taujk}]
Since the differential equation for $\tau_{j,k}$ expressed by Lemma \ref{freeito} involves only indices $(j',k')$ such that $|j'|+|k'| \leq |j|+|k|,$ we will prove the conjunction of \eqref{main jk} and \eqref{main j} by induction on $|j|+|k|$. It is understood that $k=-j$ in \eqref{main j}.

The symmetry properties of $\tau_{j,k}$ allow us to restrict ourselves to the two cases where $j,k\geq 0$ and $j> 0,\,k<0$. We may also assume that $j+k\geq 0$.

The smallest possible value of $|j|+|k|$ is $1$. So, we start with $\tau_{1,0}(T) = T \mu_1(T)=Te^{-\frac{T}{2}}$, which is smaller than $e^{-\frac{T}{4}}$ for $T$ larger than $T_{0}$. Hence, if $ |j|+|k|=1$ and $T\geq T_{0}$, then $|\tau_{j,k}(T)| \leq e^{-\frac{T}{4}}$. This proves the result when $|j|+|k|=1$.

Let us consider now $j$ and $k$ and assume that \eqref{main jk} and \eqref{main j} have been proved for all $j',k'$ such that $|j'|+|k'|<|j|+|k|$. Let us first assume that $j+k\neq 0$. In this case, define
\[ \rho_{j,k}(T) = e^{\frac{|j|+|k|}{2}T} \tau_{j,k}(T).\]
Then Lemmas \ref{bdmom} and \ref{freeito} and the induction hypothesis imply the inequality
\begin{align*}
\left\vert  \dot\rho_{j,k}(T) \right\vert  \leq & e^{\frac{|j|+|k|}{2}T} e^{-\frac{|j+k|}{3}T}   + 4 e^{\frac{|j|+|k|}{2}T}\sum_{l=1}^{|j|-1} (|j|-l) e^{-l\frac{T}{3}}  \frac{e^{-|{\rm sgn}(j)(|j|-l)+k|\frac{T}{4}}}{|j|-l+|k|}\\
& + (|j|+|k|-1) T_0 e^{\frac{|j|+|k|}{4}(T+T_0)} \sum_{l=1}^{|j|-1} (|j|-1) e^{-l \frac{T}{3}} e^{l \frac{T-T_0}{4}} \\
& + 4 e^{\frac{|j|+|k|}{2}T} \sum_{m=1}^{|k|-1} (|k|-m) e^{-m\frac{T}{3}}  \frac{e^{-|j+{\rm sgn}(k)(|k|-m)|\frac{T}{4}}}{|j|+|k|-m}\\
& + (|j|+|k|-1) T_0 e^{\frac{|j|+|k|}{4}(T+T_0)} \sum_{m=1}^{|k|-1} (|k|-1) e^{-m \frac{T}{3}} e^{m \frac{T-T_0}{4}}.
\end{align*}
Since $|j|-l\leq |j|-l+|k|$, $|k|-m\leq |j|+|k|-m$ and $e^{-l \frac{T_0}{4}}\leq 1$, we find
\begin{align}
|\dot\rho_{j,k}(T)|\leq e^{\frac{|j|+|k|}{2}T} e^{-\frac{|j+k|}{3}T} &+ 4 e^{\frac{|j|+|k|}{2}T} \sum_{l=1}^{|j|-1} e^{-l \frac{T}{3}} e^{-|{\rm sgn}(j)(|j|-l)+k|\frac{T}{4}}\nonumber \\
& + 4 e^{\frac{|j|+|k|}{2}T} \sum_{m=1}^{|k|-1} e^{-m \frac{T}{3}} e^{-|j+{\rm sgn}(k)(|k|-m)|\frac{T}{4}}\nonumber \\
& + 2  (|j|+|k|-1)^2 T_0 e^{\frac{|j|+|k|}{4}(T+T_0)} \sum_{l=1}^\infty e^{-l \frac{T}{12}}.
\label{estim jk}
\end{align}

If we are in the case where $j,k\geq 0$, then we obtain immediately the estimate
\begin{multline}\label{j>0k>0}
|\dot\rho_{j,k}(T)| \leq e^{\frac{|j|+|k|}{2}T}\\
\left(e^{-\frac{|j+k|}{3}T} +  \frac{e^{-\frac{T}{12}}}{1-e^{-\frac{T}{12}}}\left( 8 e^{-\frac{|j+k|}{4}T} + 2 (|j|+|k|-1)^2 T_0 e^{-\frac{|j|+|k|}{4}(T-T_0)} \right)\right).
\end{multline}
In the case where $j> 0$ and $k<0$, the computation is slightly more complicated. In this case, let us also assume that $j+k>0$, as we have indicated that it is possible to do. 
Then the estimation of the sum over $m$ in \eqref{estim jk} is the same as before, since 
$j+{\rm sgn}(k)(|k|-m)$ is positive for all values of $m$. However, the sign of ${\rm sgn}(j)(|j|-l)+k$ now depends on $l$. Thus, we bound the first sum over $l$ by
\[ e^{-\frac{|j+k|}{4}T} \sum_{l=1}^{j+k} e^{-l\frac{T}{12}} + \sum_{l=j+k+1}^{+\infty} e^{-l\frac{T}{3}} e^{-(l-(j+k))\frac{T}{4}}.\]
In the first term, we could actually have $e^{-l\frac{T}{3}}$ instead of $e^{-l \frac{T}{12}}$ but we are not seeking any optimality. In the second term, we write
\[e^{-(l-(j+k))\frac{T}{4}} = e^{-(2l-(j+k))\frac{T}{4}} e^{l\frac{T}{4}} \leq e^{-(j+k)\frac{T}{4}} e^{l\frac{T}{4}},\]
and we find that the first sum over $l$ in \eqref{estim jk} is bounded by $2e^{-\frac{|j+k|}{4}T}\frac{e^{-\frac{T}{12}}}{1-e^{-\frac{T}{12}}}$. Finally, we have established that, when $j>0,k<0$ and $j+k>0$,
\begin{multline*}
 |\dot\rho_{j,k}(T)| \leq e^{\frac{|j|+|k|}{2}T}\\
\left(e^{-\frac{|j+k|}{3}T} +  \frac{e^{-\frac{T}{12}}}{1-e^{-\frac{T}{12}}}\left( 12\;  e^{-\frac{|j+k|}{4}T} + 2(|j|+|k|-1)^2 T_0 e^{-\frac{|j|+|k|}{4}(T-T_0)} \right)\right).
\end{multline*}
In view of \eqref{j>0k>0}, the last estimate holds for all values of $j$ and $k$. Our choice of $T_0$ guarantees that for $T\geq T_0$, the inequalities
\[ e^{-\frac{T}{12}} + 12 \frac{e^{-\frac{T}{12}}}{1-e^{-\frac{T}{12}}} \leq 1 \; \; \; \mbox{ and } \; \; \; \frac{e^{-\frac{T}{12}}}{1-e^{-\frac{T}{12}}} \leq \frac{1}{8}\] 
hold. Hence, we find 
\[|\dot\rho_{j,k}(T)| \leq e^{\frac{|j|+|k|}{2}T} \left(  e^{-\frac{|j+k|}{4}T}+
\frac{1}{4}(|j|+|k|-1)^2 T_0 e^{-\frac{|j|+|k|}{4}(T-T_0)}
\right).\]
Integrating the last inequality from $T_0$ on and using the fact that $\frac{|j|+|k|}{2}-\frac{|j+k|}{4}\geq \frac{|j|+|k|}{4}$, we find 
\[|\rho_{j,k}(T)|\leq T_0 e^{\frac{|j|+|k|}{2}T_0} + e^{\frac{|j|+|k|}{2}T} \left(  4 \;  \frac{e^{-\frac{|j+k|}{4}T}}{|j|+|k|}+ (|j|+|k|-1) T_0 e^{-\frac{|j|+|k|}{4}(T-T_{0})}\right),\]
from which it follows immediately that
\[|\tau_{j,k}(T)|\leq 4\;  \frac{e^{-\frac{|j+k|}{4}T}}{|j|+|k|} + (|j|+|k|) T_0 e^{-\frac{|j|+|k|}{4}(T-T_0)},\]
which is the expected equality.

Let us now treat the case where $k=-j$. As before, we can assume that $j>0$. Setting 
$\rho_j(T) = e^{|j|T}\left(\tau_{j,-j}(T) -\frac{1}{|j|}\right)$, we find, using the same estimates as before, that
\[\left|\dot \rho_j(T)\right| \leq 8 e^{|j|T} \sum_{l=1}^\infty \frac{1}{2} e^{-l \frac{7T}{12}} + 2 (2|j|-1)^2 e^{\frac{|j|}{2} (T+T_0)} T_0 \sum_{l=1}^\infty e^{-l\frac{T}{12}}.\]   
It follows that
\[\left|\rho_j(T)\right| \leq T_0 e^{|j|T_0} + e^{|j|T}\frac{e^{-\frac{T}{2}}}{2\left(|j|-\frac{1}{2}\right)} + (2|j|-1) T_0 e^{\frac{|j|}{2} (T+T_0)},\]
so that
\[\left| \tau_{j,-j}(T)-\frac{1}{|j|} \right| \leq \frac{e^{-\frac{T}{2}}}{|j|} + 2 |j| T_0 e^{-\frac{|j|}{2} (T-T_0)},\]
which is the expected inequality. This concludes the proof.
\end{proof}

\subsection{Extension of the definition of the covariance}
\begin{prop}\label{convcov}
Let $f \in \H^{\frac{1}{2}}(\UC)$ be real-valued. The following properties hold.\\
\indent 1. For all $T> T_{0}$, $\displaystyle \sum_{j,k \in \Z} |jka_j(f)a_k(f) \tau_{j,k}(T)| <\infty.$\\
\indent 2. $\displaystyle  \lim_{T\to \infty} \sum_{j,k \in \Z} jka_j(f)a_k(f) \tau_{j,k}(T)  =- \Vert f\Vert^2_{\frac{1}{2}}.$
\end{prop}

\begin{proof} Choose an integer $n\geq 1$. Then for all $T\geq T_0$, Proposition \ref{taujk} implies
\begin{align*}
\sum_{|j|,|k|\leq n} |jka_j(f)a_k(f) \tau_{j,k}(T)|  &\leq \sum_{|j|,|k|\leq n} \frac{4 |jk|}{|j|+|k|} |a_j(f)a_k(f)| e^{-\frac{|j+k|}{4}T} \\
& + T_0 \sum_{ |j|,|k|\leq n}  |jk| (|j|+|k|)   |a_j(f)a_k(f)| e^{-\frac{|j|+|k|}{4}(T-T_0)}\\
& \leq  2 \sum_{|j|,|k|\leq n} \sqrt{|jk|}  |a_j(f)a_k(f)| e^{-\frac{|j+k|}{4}T}  \\
& +2 \sum_{ |j|,|k|\leq n} |j|^2 |a_j(f)| e^{-\frac{|j|}{4}(T-T_0)} |k|^2 |a_k(f)| e^{-\frac{|k|}{4}(T-T_0)} \\
&\hskip -4.5cm \leq  2 \sum_{l \in \Z} e^{-|l| \frac{T}{4}} \sum_{|j|,|k|\leq n, j+k=l} \sqrt{|jk|}  |a_j(f)a_k(f)| +2\left( \sum_{ |j|\leq n} |j|^2 |a_j(f)| e^{-\frac{|j|}{4}(T-T_0)} \right)^2 \\
&\hskip -4.5cm \leq \|f\|^2_{\frac{1}{2}} \left(2\sum_{l \in \Z} e^{-|l| \frac{T}{4}} +2  \sum_{j\in \Z} |j|^3 e^{-\frac{|j|}{2}(T-T_0)}\right). 
\end{align*}
The first assertion follows. The second is a consequence of the second statement in Proposition \ref{taujk} and the theorem of dominated convergence.
\end{proof}

Proposition \ref{convcov} above allows us to give a new definition of the covariance $\sigma_T$ when $T$ is large enough.

\begin{defi} \label{newsigma}
For all $T > T_0$ and all $f \in \H^{\frac{1}{2}}(\UC),$
we define 
$$ \sigma_T(f,f) = - \sum_{j,k \in \Z} jka_j(f)a_k(f) \tau_{j,k}(T).$$
\end{defi}

\begin{lem} Let $f$ be a function of $C^{1,1}(\UC)$. For all $T > T_0$, the two definitions (Definition \ref{sigma} and Definition \ref{newsigma}) of $\sigma_T(f,f)$ coincide.
\end{lem}

\begin{proof} The series $\sum_{j\in \Z} |a_j(f')|$ is convergent, so that 
$S_n(f^\prime)(e^{i\xi}) = i\sum_{|j|\leq n}j a_j(f) e^{ij\xi}$ converges uniformly to $f^\prime$ on $\UC$ as $n$ tends to infinity.
Therefore, starting from Definition \ref{sigma},
$$ \sigma_T(f,f) = -\int_0^T \tau\left(\sum_{j,k \in \Z} jka_j(f)a_k(f) (u_sv_{T-s})^j(u_sw_{T-s})^k\right)ds.$$
As the processes are unitary and $\sum_{j\in\Z} |j||a_j(f)| <\infty$, 
we get by dominated convergence that, for all $T \geq 0$,
\[ \sigma_T(f,f) = - \sum_{j,k \in \Z} jka_j(f)a_k(f) \tau_{j,k}(T),\]
as expected.
\end{proof}

Theorem \ref{cvtohaar} is now a straightforward consequence of the polarisation of Definition \ref{newsigma} and Proposition \ref{convcov}.

\begin{remark}
Let us emphasize that Proposition \ref{taujk} implies that, for all $\epsilon>0$ and all $T>T_{0}$, the following series converges:
\[ \sum_{j,k\in \Z} (|j|+|k|)^{1-\epsilon} |\tau_{j,k}(T)|^{2} <+\infty.\]
Hence, for all $T>T_{0}$, the equality
\[ K_{T}(e^{i\theta},e^{i\phi}) = \sum_{(j,k)\in \Z^2 \setminus\{(0,0)\}} e^{ij\theta} e^{ik\phi} \overline{\tau_{j,k}(T)}\]
defines $K_{T}$ as a square-integrable real-valued function on $\UC^2$ and, for all $\epsilon >0$ and $f,g \in H^{\frac{3}{4}+\epsilon}(\UC)$, one has the equality
\[\sigma_{T}(f,g)=\int_{[0,2\pi]^2} f'(e^{i\theta}) K_{T}(e^{i\theta},e^{i\phi}) g'(e^{i\phi}) \frac{d\theta d\phi}{4\pi^{2}}.\]
\end{remark}


We conclude this study of the covariance by showing some puzzling numerical experiments (see Figure \ref{covariances}). It is striking on these pictures that the behaviour of the covariance $\sigma_T(f,g)$ is complicated and interesting for small $T$, and much simpler for large $T$. It is thus not surprising that we have been only able to analyse $\sigma_T$ for large $T$. 

\begin{figure}[ht!]
\begin{center}
\includegraphics[width=3.5cm]{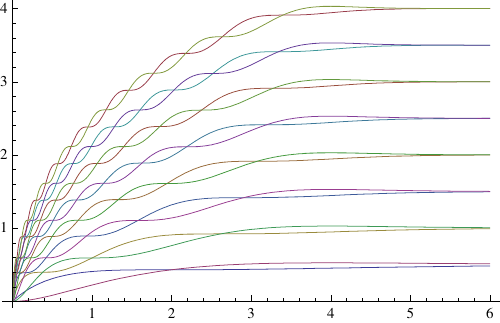} \hspace{0.3cm} \includegraphics[width=3.5cm]{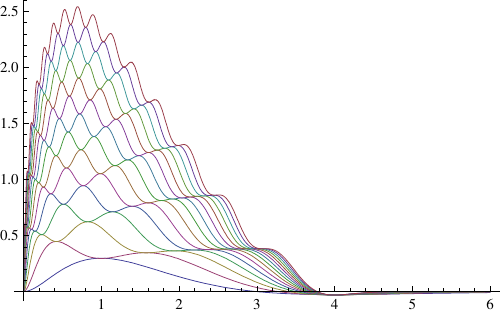}\hspace{0.3cm}\includegraphics[width=3.5cm]{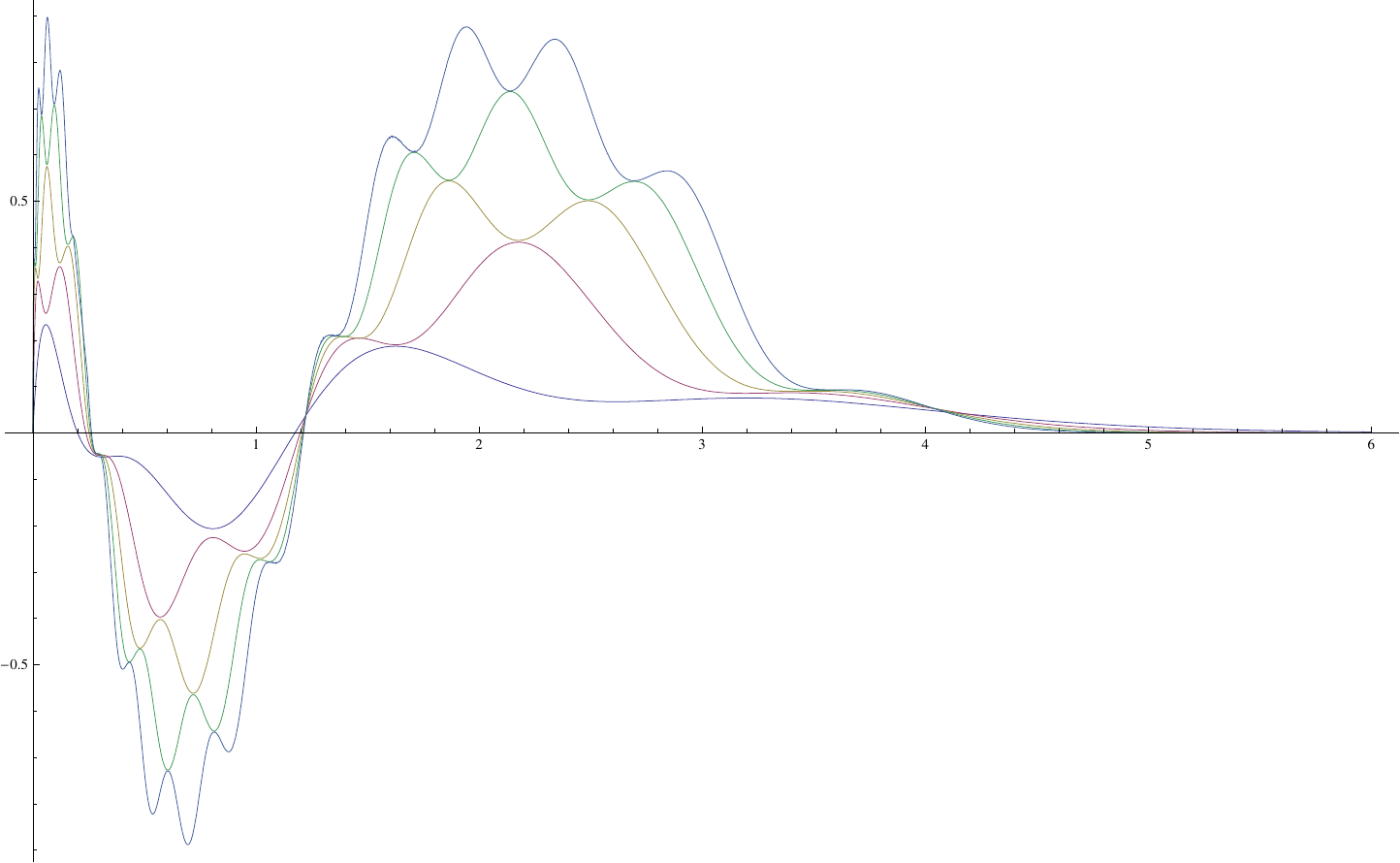} \\
\vskip 3mm
\includegraphics[width=3.5cm]{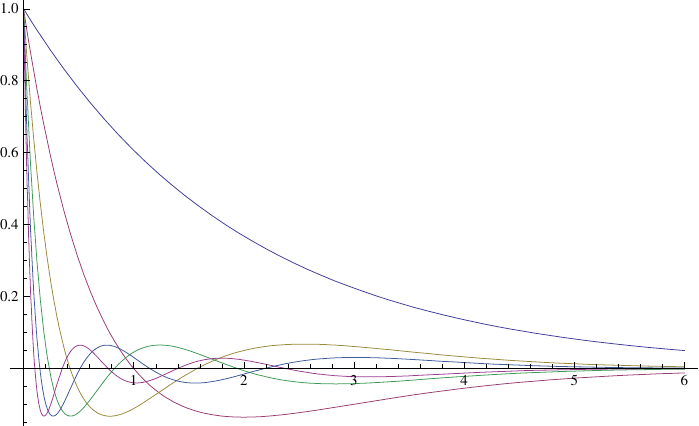} \hspace{0.3cm} \includegraphics[width=3.5cm]{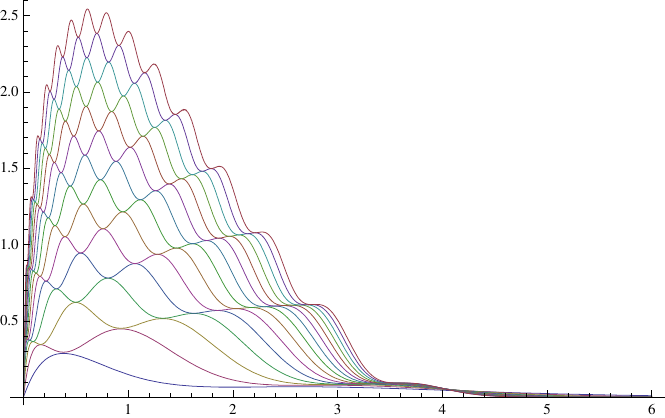}\hspace{0.3cm} \includegraphics[width=3.5cm]{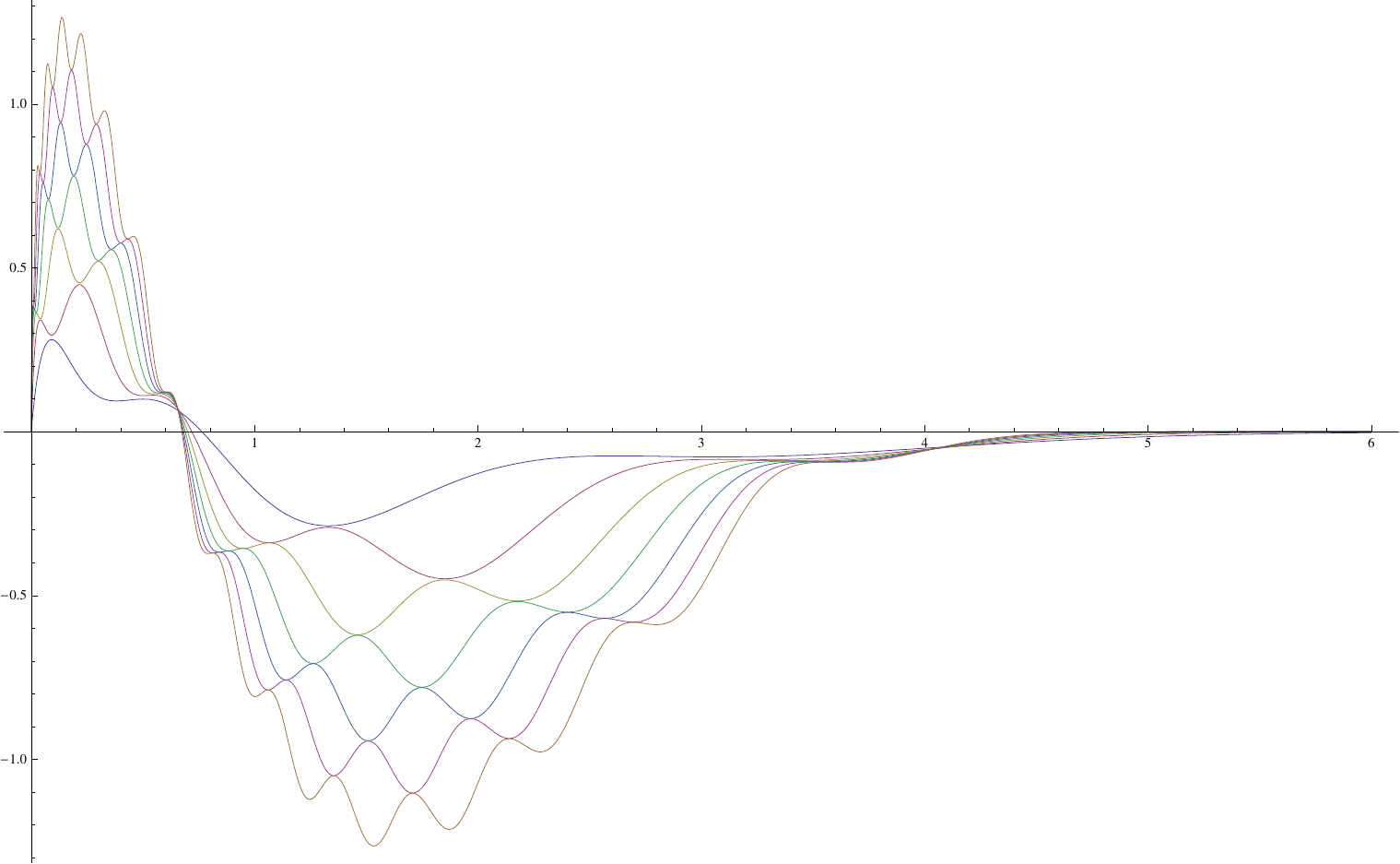}
\end{center}
\caption{\label{covariances} For all $k\geq 1$, let us define $s_{k}(e^{i\theta})=\sin (k\theta)$ and $c_{k}(e^{i\theta})=\cos (k\theta)$. The pictures above are the graphs of the following functions of $T$ for $T\in [0,6]$. Top left : $\sigma_{T}(s_{k},s_{k})$ and $\sigma_{T}(c_{k},c_{k})$ for $k\in \{1,\ldots,8\}$. Bottom left : $\mu_{k}(T)$ for $k\in \{1,\ldots,6\}$. Top center : $\sigma_{T}(s_{k},s_{k+1})$ for $k\in \{1,\ldots,15\}$. Bottom center : $\sigma_{T}(c_{k},c_{k+1})$ for $k\in \{1,\ldots,15\}$. Top right : $\sigma_{T}(s_{k},s_{k+3})$ for $k\in \{1,4,7,10,13\}$. Bottom right : $\sigma_{T}(s_{k},s_{k+2})$ for odd $k\in \{1,\ldots,13\}$.}
\end{figure}

\section{Combinatorial approaches} \label{sec combi}

\subsection{The differential system satisfied by the $\tau_{j,k}$}

The differential system satisfied by the functions $\tau_{j,k}$ (Lemma \ref{freeito}) can be interpreted, at least when $j$ and $k$ have the same sign, in terms of enumeration of walks on the symmetric group, in the same vein as the computations made by one of us in \cite{L}. This is what we explain in this section.

Fix $j \geq 1.$ We consider the Cayley graph on the symmetric group $\mathfrak S_j$
generated by all transpositions. The vertices of this graph are the elements of $\mathfrak S_j$
and two permutations $\sigma_1$ and $\sigma_2$ are joined by an edge if  and only if
$\sigma_1 \sigma_2^{-1}$ is a transposition. A finite sequence $(\sigma_0, \ldots, \sigma_n)$ of permutations such that $\sigma_i$ and $\sigma_{i+1}$ are joined by an edge for all
$i \in \{0, \ldots, n-1\}$ is called a path of length  $n.$ The distance between two permutations is the length of the shortest path that joins them. 
We call defect of a path the number of steps in the path which increase the distance to identity.  Heuristically, one can understand the defect as follows :
each time we compose a permutation with a transposition, either we cut a cycle into two pieces and this is a step which decreases the distance to identity, or we coalesce two cycles into a bigger one and this is a step which increases the distance to identity. The defect counts the number of steps of the second kind.

For any $\sigma \in \mathfrak S_j,$ and  any two integers $n,d \geq 0,$ we denote by 
$S(\sigma, n, d)$ the number of paths in the Cayley graph of $\mathfrak S_j$ starting from $\sigma,$
of length $n$ and with defect $d.$ The interested reader can find more details about those 
combinatorial objects in \cite{L}.

Let $j, k \geq 1.$ If $\sigma \in \mathfrak S_j$ and $\tau \in \mathfrak S_k,$ we denote
by $\sigma \times \tau$ the concatenation of $\sigma$ and $\tau,$ that is the  permutation
in $\mathfrak S_{j+k}$ such that $\sigma \times \tau(i) = \sigma(i)$ if $1 \leq i\leq j$
and $\sigma \times \tau(i) = \tau(i-j)+j$ if $j+1 \leq i\leq j+k.$

From Theorem 3.3 in \cite{L}, it follows that for all $T\geq 0$,
\begin{multline} \label{devcombi}
\E\left[\tr(U_N(T)^j)\tr(U_N(T)^k)\right]- \E\Big[\tr(U_N(T)^j)\Big] \E\left[\tr(U_N(T)^k)\right]  \\
= e^{-(j+k)\frac{T}{2}}\left(\sum_{n,d=0}^\infty \frac{(-T)^n}{n!N^{2d}}S((1 \ldots j)\times(1 \ldots k),n,d) \right. \\
\left.- \sum_{n_1,n_2,d_1,d_2=0}^\infty \frac{(-T)^{n_1+n_2}}{n_1!n_2!N^{2(d_1+d_2)}}S((1 \ldots j),n_1,d_1)S((1 \ldots k),n_2,d_2)\right).
\end{multline}
Moreover, for all $T'\geq 0,$ we recall that all the expansions involved converge uniformly on $(N,T) \in \N\times [0,T']$. 

Using this equality, it is for example easy to check that 
\begin{multline*}
\lim_{N \ra \infty}  \left(\E\left[\tr(U_N(T)^j\tr(U_N(T)^k)\right]- \E\Big[\tr(U_N(T)^j)\Big]\E\left[\tr(U_N(T)^k)\right]\right) = \\
e^{-(j+k)\frac{T}{2}} \sum_{n=0}^\infty \frac{(-T)^n}{n!} 
\bigg( S((1 \ldots j)\times(1 \ldots k),n,0)  \bigg.\\
\bigg. - \sum_{n_1=0}^n \binom{n}{n_1} 
S((1 \ldots j),n_1,0)S((1 \ldots k),n-n_1,0)\bigg) =0,
\end{multline*}
where the last equality comes from Proposition 5.3 of \cite{L}. Each term of the sum is indeed zero and heuristically, it means that a path without defect starting from $(1 \ldots j)\times(1 \ldots k)$ is simply obtained by ``shuffling'' two paths without defect from each of the two cycles
in their respective symmetric group.

More interesting for us is the fact we can also deduce from \eqref{devcombi} that
\begin{multline}
\kappa_{j,k}(T)\build{=}_{}^{\rm{(def)}}\\
\lim_{N \ra \infty}  N^{2} \left( \E\left[\tr(U_N(T)^j\tr(U_N(T)^k)\right]- \E\Big[\tr(U_N(T)^j)\Big]\E\left[\tr(U_N(T)^k)\right]\right) =  \\
e^{-(j+k)\frac{T}{2}} \sum_{n=0}^\infty \frac{(-T)^n}{n!}  S^\prime((1 \ldots j)\times(1 \ldots k),n,1) , \label{sigmanm}
\end{multline}
where, $\sigma \in \mathfrak S_j, \tau \in \mathfrak S_k$ and $n \geq 1$ being given,
we use the notation 
\begin{multline*}
S^\prime(\sigma\times\tau,n,1)= S(\sigma\times\tau,n,1)\\
- \sum_{n_1=0}^n \binom{n}{n_1} \bigg(S(\sigma,n_1,1)S(\tau,n-n_1,0) + S(\sigma,n_1,0)S(\tau,n-n_1,1)\bigg).
\end{multline*}
Thus defined, $S^\prime(\sigma\times\tau,n,1)$ is the number of paths of length $n$ starting from $\sigma\times\tau$ such that the unique step which increases the distance to the identity is the multiplication by a transposition which exchanges an element of $\{1,\ldots,j\}$ with an element of $\{j+1,\ldots,j+k\}$. Thus, heuristically, the unique step which is a coalescence is a coalescence between $\sigma$ and $\tau$.

Our goal is now to show the following combinatorial identity
\begin{prop} \label{itocombi}
For any integers $j,k \geq 1,$ and $n \geq 0,$ we have
\begin{eqnarray*}
S^\prime((1\ldots j) \!\!\!\!\!\! &\times  \!\!\!\!\!& (1\ldots k),n+1,1) = 
jk \,\, S((1 \ldots j+k),n,0) \\
&+ j& \!\!\sum_{l=1}^{j-1} \sum_{p=0}^n \binom{n}{p}S((1 \ldots l),p,0) S^\prime((1\ldots j-l)\times(1\ldots k),n-p,1) \\
&+k& \!\! \sum_{m=1}^{k-1} \sum_{q=0}^n \binom{n}{q} S((1 \ldots m),q,0) S^\prime((1\ldots j)\times(1\ldots k-m),n-q,1).
\end{eqnarray*}

\end{prop}

The combinatorial interpretation of this identity is the following : let us consider 
a path of length $n+1$ from 
$(1\ldots j)\times(1\ldots k)$ whose unique step increasing the distance to identity is a true coalescence between the two 
cycles. 
The first step of such a path can be of three kinds, corresponding respectively to the three
terms of the right hand-side :
\begin{itemize}
\item either it coalesces the cycles, creating a $(j+k)$-cycle, and this can be done
by choosing an element in each cycle. Then the path can be completed by any path of length $n$ without defect from a $(j+k)$-cycle.
\item  either it cuts the cycle $(1 \ldots j)$ into two cycles, one of length $l$ that will then be cut $p$ times without being affected by the coalescence and another of length $j-l$ which contains the element which will be exchanged with an element of $\{j+1,\ldots,j+k\}$ during the coalescing step. 
\item either, symmetrically, it cuts the cycle $(1 \ldots k)$.
\end{itemize}

We will hereafter propose a  rigorous proof of this identity through the free stochastic calculus tools introduced above in the paper. It should be noted that the combinatorics which we investigate here is related to that of annular noncrossing partitions introduced by J. Mingo and A. Nica \cite{MingoNica}.
\begin{proof}
Let the integers $j,k \geq 1$ and the real $T \geq 0$ be fixed. If we consider the quantities $\kappa_{j,k}(T)$ as defined in \eqref{sigmanm}, if we denote, for any $r\in \Z,$ by $f_r : \UC 
\ra \C$ the function given by $f_r(z)=z^r,$ then, from Definition \ref{sigma} and Theorem \ref{main theo}, we get
$\kappa_{j,k}(T) = \sigma_T(f_j,f_k)$ and from \eqref{deftaujk}, it can be reexpressed as
$\kappa_{j,k}(T) = -jk\,\, \tau_{j,k}(T).$ Now, from Lemma \ref{freeito}, we get immediately 
\begin{multline*}
\dot\kappa_{j,k}(T) =  -jk\,\mu_{j+k}(T) - \frac{j+k}{2}\, \kappa_{j,k}(T) \\
- j\sum_{l=1}^{j-1}  \mu_l(T) \sigma_{j-l,k}(T)  - k\sum_{m=1}^{k-1}  \mu_m(T) \kappa_{j,k-m}(T),
\end{multline*}
so that we get immediately the anounced result, as we know from \cite{L}
that, for any $r \in \N^*,$
$$ \mu_{r}(T) = e^{-r \frac{T}{2}} \sum_{n=0}^\infty \frac{(-T)^n}{n!}S((1 \ldots r),n,0)$$
and from \eqref{sigmanm} that
$$ \dot\kappa_{j,k}(T) = - \frac{j+k}{2} \kappa_{j,k}(T) - e^{-(j+k)\frac{T}{2}}
\sum_{n=0}^\infty \frac{(-T)^n}{n!}S^\prime((1\ldots j)\times(1 \ldots k),n+1,1).$$
\end{proof}

\subsection{Mixed moments of special unitary matrices}
\label{combi moments}

In principle, any computation involving functions invariant by conjugation on the unitary group can be performed by using harmonic analysis, that is, the representation theory of the unitary group. In this section, we use this approach to prove the following formula, which yields for each $N\geq 3$ an explicit expression for the covariance of traces of powers of the Brownian motion on $\SU(N)$. With the help of Section \ref{u su}, it is easy to deduce the analogous result for the Brownian motion on $\U(N)$.

\begin{theo}\label{covariance traces} Let $N\geq 3$ be an integer. Consider, on $\SU(N)$, the Brownian motion $(V_{N}(t))_{t\geq 0}$ associated with the scalar product $\langle X,Y\rangle_{\su(N)} = N \Tr(X^{*}Y)$ on $\su(N)$. Let $n$ and $m$ be positive integers. Assume that $N\geq n+m+1$. Then
\begin{align*}
\E\left[\Tr(V_{N}(t)^{n}) \overline{\Tr(V_{N}(t)^{m})}\right] &= n\delta_{n,m}  + (-1)^{n+m} e^{-(n+m)\frac{t}{2} - \frac{n(n-1)+m(m-1)}{N}\frac{t}{2}-\frac{(n-m)^{2}}{N^{2}}\frac{t}{2}}\\ & \hspace{-3cm} \sum_{r_{1}=0}^{n-1}\sum_{r_{2}=0}^{m-1} \left [(-1)^{r_{1}+r_{2}} e^{-nr_{1} \frac{t}{2}} \binom{n-1}{r_{1}}\binom{N+r_{1}}{n} e^{-nr_{2} \frac{t}{2}} \binom{m-1}{r_{2}}\binom{N+r_{2}}{m} \right.\\
& \hspace{1cm} \left.\frac{(N+r_{1}+r_{2}+1)(N-n-m+r_{1}+r_{2}+1)}{(N-n+r_{1}+r_{2}+1)(N-m+r_{1}+r_{2}+1)}\right].
\end{align*}
\end{theo}

The basic strategy for the proof is to expand the heat kernel and the traces in the basis of Schur functions, and then to use the multiplication rules for Schur functions and their orthogonality properties. The multiplication rules are expressed by the Littlewood-Richardson formula and they are rather complicated. Fortunately, in the present situation, the Young diagrams which occur are simple enough for the computation to be tractable. 



Let us recall the fundamental facts about Schur functions. Details can be found in \cite{FultonHarris}. A Young diagram is a non-increasing sequence of non-negative integers. If $\lambda=(\lambda_{1}\geq \ldots \ldots \lambda_{k}>0)$ is such a sequence, we call $k$ the length of $\lambda$ and denote it by $\ell(\lambda)$. The set of Young diagrams of length at most $k$ is denoted by $\N^{k}_{\downarrow}$. We draw Young diagrams downwards in rows, according to the convention illustrated by the left part of Figure \ref{Young1}.

\begin{figure}[h!]
\begin{center}
\includegraphics{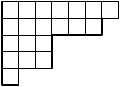} \hspace{2cm} \includegraphics{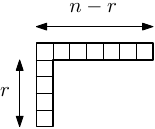}
\end{center}
\caption{\label{Young1} The Young diagram on the left is $(7,6,3,3,1)$, which we also denote by $7\,6\,3^{2}\,1$. The diagram on the right is $\eta_{n,r}=(n-1) \, 1^r$.} 
\end{figure}

The Schur function $s_{\lambda}$ is a symmetric function which, when evaluated on strictly less than $\ell(\lambda)$ variables, yields 0. Whenever $\ell(\lambda)\leq N$, the function $s_{\lambda}$ is well defined and non-zero on $\SU(N)$. Its value $s_{\lambda}(I_{N})$ at the identity matrix in particular is a positive integer, which is the dimension of the irreducible representation of $\SU(N)$ of which $s_{\lambda}$ is the character. Another number attached to $\lambda$ will play an important role for us, which is the non-negative real number $c(\lambda)$ such that $\Delta s_{\lambda}=-c(\lambda)s_{\lambda}$. 

It happens that distinct Young diagrams yield the same function on $\SU(N)$~: if $\lambda$ and $\mu$ are Young diagrams such that $\ell(\lambda),\ell(\mu)\leq N$, then $s_{\lambda}=s_{\mu}$ if and only if there exists $l\in \Z$ such that $\lambda=\mu+(l,\ldots, l)=\mu + l^{N}$. In fact, if $\rho_{\lambda}$ and $\rho_{\mu}$ are the representations of $\U(N)$ corresponding to $\lambda$ and $\mu$, then $\rho_{\lambda}=\rho_{\mu}\otimes \det^{\otimes l}$ and the restrictions of these representations to $\SU(N)$ are equal.

Finally, we need to use the decomposition of the heat kernel and the function $U\mapsto \Tr(U^{n})$ in terms of Schur functions. For the latter, we introduce a class of Young diagrams called hooks. For all $n\geq 1$ and all $r\in \{0,\ldots,n-1\}$, we define
\[\eta_{n,r}=(n-r,\underbrace{1,\ldots,1}_{r})=(n-r) \, 1^{r},\]
which is depicted on the right part of Figure \ref{Young1}.

The heat kernel at time $t$ on $\SU(N)$ is the density, denoted by $Q_{t}:\SU(N)\to \R$, of the distribution of $V_{N}(t)$ with respect to the Haar measure.

\begin{prop} Choose $N\geq 1$ and $U\in \SU(N)$. Then the following equalities hold.\\
1. For all $n\geq 1$, $\displaystyle\Tr(U^{n})=\sum_{r=0}^{n-1} (-1)^{r} s_{\eta_{n,r}}(U)$.\\
2. For all $t\geq 0$, $\displaystyle Q_{t}(U)=\sum_{\lambda\in \N^{N-1}_{\downarrow}} e^{-\frac{c(\lambda)}{2}t} s_{\lambda}(I_{N})s_{\lambda}(U)$.
\end{prop}

The proof of the first equality can be found in \cite{Macdonald}, the proof of the second in \cite{Liao}. The expectation that we want to compute in order to prove Theorem \ref{covariance traces} is thus equal to
\begin{multline*}
\E\left[\Tr(V_{N}(t)^{n}) \overline{\Tr(V_{N}(t)^{m})}\right] =  \sum_{\lambda\in \N^{N-1}_{\downarrow}} e^{-\frac{c(\lambda)}{2}t} s_{\lambda}(I_{N}) \sum_{r_{1}=0}^{n-1}\sum_{r_{2}=0}^{m-1} (-1)^{r_{1}+r_{2}}\\ \int_{\SU(N)} s_{\lambda}(U)s_{\eta_{n,r_{1}}}(U) \overline{s_{\eta_{m,r_{2}}}(U)}\; dU.
\end{multline*}

The multiplication of Schur functions is governed by the Littlewood-Richardson formula, which describes a non-negative integer $N_{\alpha, \beta}^{\gamma}$ for each triple of Young diagrams $\alpha,\beta,\gamma$, in such a way that
\[s_{\alpha}s_{\beta}=\sum_{\gamma} N_{\alpha, \beta}^{\gamma} s_{\gamma}.\]
Using these coefficients, the integral above can be rewritten as 
\begin{align*}
\int_{\SU(N)} s_{\lambda}(U)s_{\eta_{n,r_{1}}}(U) \overline{s_{\eta_{m,r_{2}}}(U)}\; dU &= 
\sum_{\gamma} N_{\lambda, \eta_{n,r_{1}}}^{\gamma} \int_{\SU(N)}s_{\gamma}(U) \overline{s_{\eta_{m,r_{2}}}(U)}\; dU \\
&= \sum_{\gamma} N_{\lambda, \eta_{n,r_{1}}}^{\gamma} \sum_{l\geq 0} {\mathbbm 1}_{\gamma =\eta_{m,r_{2}}+l^{N}}\\
& = \sum_{l\geq 0} N_{\lambda, \eta_{n,r_{1}}}^{\eta_{m,r_{2}}+l^{N}}.
\end{align*}
Thus, we need to compute 
\begin{equation}\label{combi lambda}
\sum_{r_{1}=0}^{n-1}\sum_{r_{2}=0}^{m-1} (-1)^{r_{1}+r_{2}}  \sum_{l\geq 0} N_{\lambda, \eta_{n,r_{1}}}^{\eta_{m,r_{2}}+l^{N}}.
\end{equation}
It turns out that a slightly more general computation is simpler to perform : we compute the Littlewood-Richardson coefficient $N_{\alpha,\eta_{n,r}}^{\beta}$ for all $\alpha,\beta$ and all $n,r$. Let us introduce some notation.

Let $\alpha=(\alpha_{1},\ldots)$ and $\beta=(\beta_{1},\ldots)$ be two Young diagrams. 
Set $|\alpha|=\sum_{i} \alpha_{i}$ and $|\beta|=\sum_{i} \beta_{i}$. We assume that $\alpha\subset \beta$, that is, $\alpha_{i}\leq \beta_{i}$ for all $i$. Then we denote by $\beta/\alpha$ the set of boxes of the graphical representation of $\beta$ which are not contained in $\alpha$. We say that a subset of $\beta/\alpha$ is connected if one can go from any box to any other inside this subset by a path which jumps from a box to another only when they share an edge. 

We denote by $k(\beta/\alpha)$ the number of connected components of $\beta/\alpha$.
Also, we define $v(\beta/\alpha)$ as the number of boxes of $\beta/\alpha$ which are such that the box located immediately above also belongs to $\beta/\alpha$. Alternatively, this is the number of distinct occurrences of the motif formed by two consecutive boxes one above the other in $\beta/\alpha$.

Our main combinatorial result is the following. 

\begin{prop} Let $\alpha$ and $\beta$ be two Young diagrams. Let $\eta_{n,r}$ be a hook. Then $N_{\alpha,\eta_{n,r}}^{\beta}$ is non-zero if and only if the following conditions are satisfied : $\alpha\subset \beta$, $|\beta|=|\alpha|+n$, $\beta/\alpha$ contains no $2\times 2$ square, and $v(\beta/\alpha)\leq r \leq v(\beta/\alpha)+k(\beta/\alpha)-1$. In this case, $N_{\alpha,\eta_{n,r}}^{\beta}=\binom{k(\beta/\alpha)-1}{r-v(\beta/\alpha)}$.
\end{prop}

\begin{proof} According to the Littlewood-Richardson rule, $N_{\alpha,\eta_{n,r}}^{\beta}$ is the number of strict expansions of $\alpha$ by $\eta_{n,r}$ which yield $\beta$, that is, the number of fillings of $\beta/\alpha$ with the boxes of $\eta_{n,r}$ such that the following conditions are satisfied:\\
1. for all $s\geq 1$, the union of $\alpha$ and the boxes of $\beta/\alpha$ filled by the first $s$ rows of $\eta_{n,r}$ is a Young diagram,\\
2. no two boxes of the first row of $\eta_{n,r}$ are put in the same column of $\beta/\alpha$,\\
3. if one goes through the boxes of $\beta/\alpha$ from right to left and from top to bottom, writing for each box the number of the row of $\eta_{n,r}$ from which is issued the box which has been used to fill it, one obtains a sequence which starts with $1$, and in which all other numbers $2,\ldots,r$ appear, not necessarily consecutively, in this order. 

It is important to notice that, according to the third rule, a strict expansion of $\alpha$ by a hook which yields $\beta$ is completely characterized by the set of boxes of $\beta/\alpha$ which are filled by boxes issued from the first row of the hook. We say for short that these boxes of $\beta/\alpha$ are {\em filled by the first row}.

The first two conditions $\alpha\subset \beta$ and $|\beta|=|\alpha|+n$ are obviously implied by this rule. A less trivial implication is that there cannot exist a strict expansion if $\beta/\alpha$ contains a $2\times 2$ square. Indeed, by the first two rules, the bottom-left box of the square cannot be filled by the first row and the bottom-right box must then be filled with a boxed issued from a strictly lower (in the graphical representation) row of $\eta_{n,r}$. This contradicts the third rule.

Let us assume that $\beta/\alpha$ contains no $2\times 2$ square. Then each connected component of $\beta/\alpha$ is a ``snake'' (see Figure \ref{snake}).
 \begin{figure}[ht!]
\begin{center}
\includegraphics{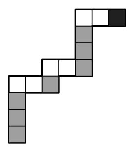}
\end{center}
\caption{\label{snake}White boxes must be filled by boxes issued from the first row of $\eta_{n,r}$. Grey boxes cannot. The black box may or may not, except if this snake is the topmost connected component of $\beta/\alpha$, in which case it must also be filled by a box issued from the first row of $\eta_{n,r}$.}
\end{figure}

Any box of such a snake which has a box on its right must be filled by the first row. These boxes are the white boxes in Figure \ref{snake}. Any box located below a white box cannot be filled by the first row. These boxes are the grey boxes in Figure \ref{snake}. Only one box is not in one of these two cases, the top-right box of the snake. In the topmost connected component of $\beta/\alpha$ the third rule implies that this box must be filled by the first row.

Finally, if the first three conditions are satisfied, then $\beta/\alpha$ contains one box in each connected component, except the topmost one, which can either be filled by the first row or not. The minimal number of boxes which are not filled by the first row is the number of grey boxes, which we have denoted by $v(\beta/\alpha)$. This is the minimal value of $r$ for which there exists a strict expansion of $\alpha$ by $\eta_{n,r}$ which yields $\beta$. Moreover, for this value of $r$, the expansion is unique, since the boxes filled by the first row are completely determined. Similarly, the maximal value of $r$ is $v(\beta/\alpha)+k(\beta/\alpha)-1$. For $r$ between these two bounds, there are exactly as many expansions as there are choices of which snakes have their top-right box filled by the first row. There are thus $\binom{k(\beta/\alpha)-1}{r-v(\beta/\alpha)}$ such expansions.
\end{proof}

\begin{cor}\label{alternate LR} Let $\alpha$ and $\beta$ be two Young diagrams. Choose $n\geq 1$. Then
\[\sum_{r=0}^{n-1} (-1)^{r} N_{\alpha,\eta_{n,r}}^{\beta} = (-1)^{v(\beta/\alpha)}\]
if $\alpha\subset \beta$, $|\beta|=|\alpha|+n$, $\beta/\alpha$ contains no $2\times 2$ square and is connected. Otherwise, it is equal to 0.
\end{cor}

\begin{proof} If the first three conditions are not satisfied, then $N_{\alpha,\eta_{n,r}}^{\beta}=0$ for all $r=0\ldots n-1$. Let us assume that they are satisfied. Then, by the previous proposition, the sum above is equal to
\[\sum_{r=v(\beta/\alpha)}^{v(\beta/\alpha)+k(\beta/\alpha)-1} (-1)^{r}\binom{k(\beta/\alpha)-1}{r-v(\beta/\alpha)},\]
which is equal to $0$ unless $k(\beta/\alpha)=1$. In this case, only one term of the sum is non-zero, for $r=v(\beta/\alpha)$.
\end{proof}

We apply now this result when $\beta$ is of the sum of a hook and a rectangle.

\begin{lem} Consider $n\geq m\geq 1$, $r_{2}\in \{0,\ldots,m-1\}$, and $N\geq m+n$. 
For all $r_{1}\in \{0,\ldots,n-1\}$, define
\[\lambda^{N}_{m,r_{2},n,r_{1}}=(n-r_{1}+m-r_{2})\, (n-r_{1}+1)^{r_{2}} \, (n-r_{1})^{N-r_{1}-r_{2}-2}\, (n-r_{1}-1)^{r_{1}}.\]
Then, for all $\lambda\in \N^{N-1}_{\downarrow}$ and all $l\geq 1$,
\[\sum_{r_{1}=0}^{n-1} (-1)^{r_{1}} N_{\lambda,\eta_{n,r_{1}}}^{\eta_{m,r_{2}}+l^{N}} =\left\{\begin{array}{cl} (-1)^{n-l} & \mbox{ if } l\in \{1,\ldots,n\} \mbox{ and } \lambda=\lambda^{N}_{m,r_{2},n,n-l},\\ 0 & \mbox{ otherwise.}\end{array}\right.\]
Moreover, when $l\in \{1,\ldots,n\}$, the only non-zero term of the sum is the term corresponding to $r_{1}=n-l$.

Finally, if $n=m$, then $N_{\lambda,\eta_{n,r_{1}}}^{\eta_{m,r_{2}}}=1$ if $r_{1}=r_{2}$ and $\lambda$ is the empty diagram, and $0$ otherwise.
\end{lem}

\begin{figure}[ht!]
\begin{center}
\scalebox{0.8}{\includegraphics{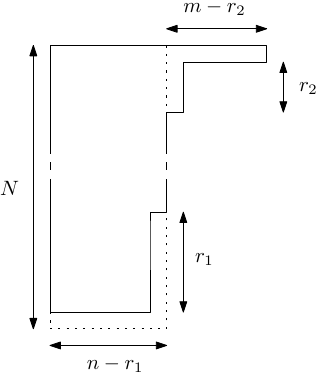}}
\end{center}
\caption{The diagram $\lambda^{N}_{m,r_{2},n,r_{1}}$.}
\end{figure}

\begin{proof} Let us first consider the case $n>m$. In this case, according to  Corollary \ref{alternate LR}, in order for the sum to be non-zero, $\lambda$ must be a Young diagram of length at most $N-1$, contained in $\eta_{m,r_{2}}+l^{N}$, such that $(\eta_{m,r_{2}}+l^{N})/\lambda$ contains no $2\times 2$ square and is connected. Since $n>m$, $l$ must be positive, so that the diagram $\eta_{m,r_{2}}+l^{N}$ has length $N$ whereas $\lambda$ has length at most $N-1$. Thus, the $N$-th row of $(\eta_{m,r_{2}}+l^{N})/\lambda$ is not empty, it has actually length $l$. In particular, $|\eta_{m,r_{2}}+l^{N}|-|\lambda|\geq l$. If $l>n$, all the Littlewood-Richardson coefficients appearing in the sum are zero. Otherwise, if $l\leq n$, there is exactly one way to choose $\lambda$ a subdiagram of $\eta_{m,r_{2}}+l^{N}$ such that all conditions are satisfied : it is $\lambda=\lambda^{N}_{m,r_{2},n,n-l}$. 

When $n=m$, nothing changes for $l\geq 1$. However, the sum may be non-zero even for $l=0$. The diagram $\lambda$ must be the empty diagram and it is easy to check that $N_{\varnothing,\eta_{n,r_{1}}}^{\eta_{m,r_{2}}}=\delta_{n,m}\delta_{r_{1},r_{2}}$.
\end{proof}

We can now go on to compute \eqref{combi lambda}. We find the following result.

\begin{prop}\label{cov but c and s} Let $N$, $n$ and $m$ be three positive integers. Assume that $n\geq m$ and $N\geq n+m+1$. Then
\[\E\left[\Tr(V_{N}(t)^{n}) \overline{\Tr(V_{N}(t)^{m})}\right] = n\delta_{n,m}+\sum_{r_{1}=0}^{n-1}\sum_{r_{2}=0}^{m-1} (-1)^{r_{1}+r_{2}} e^{-\frac{c(\lambda^{N}_{m,r_{2},n,r_{1}})}{2}t} s_{\lambda^{N}_{m,r_{2},n,r_{1}}}(I_{N}).\] 
\end{prop}

\begin{proof} We have
\begin{align*}
\sum_{r_{1}=0}^{n-1}\sum_{r_{2}=0}^{m-1} (-1)^{r_{1}+r_{2}}  \sum_{l\geq 0} N_{\lambda, \eta_{n,r_{1}}}^{\eta_{m,r_{2}}+l^{N}}&=n\delta_{m,n}{\mathbbm 1}_{\lambda=\emptyset}+ \sum_{r_{2}=0}^{m-1} (-1)^{r_{2}}\sum_{l=1}^{n} (-1)^{n-l} {\mathbbm 1}_{\lambda=\lambda^{N}_{m,r_{2},n,n-l}}\\
&=n\delta_{m,n}{\mathbbm 1}_{\lambda=\emptyset}+\sum_{r_{1}=0}^{n-1}\sum_{r_{2}=0}^{m-1} (-1)^{r_{1}+r_{2}} {\mathbbm 1}_{\lambda=\lambda^{N}_{m,r_{2},n,r_{1}}}.
\end{align*}
The claimed equality follows easily.
\end{proof}

In order to prove Theorem \ref{covariance traces}, there remains to compute $c(\lambda^{N}_{m,r_{2},n,r_{1}})$ and $s_{\lambda^{N}_{m,r_{2},n,r_{1}}}(I_{N})$. This is by no means complicated but slightly tedious. We recall the general formulae, give the results in this particular case and invite the reader to check them if s/he feels inclined to do so.

\begin{lem}\label{c and s} Consider $n,m\geq 1$, $r_{1}\in \{0,\ldots,n-1\}$ and $r_{2}\in \{0,\ldots,m-1\}$. Then the following identities hold.
\[
c(\lambda^{N}_{m,r_{2},n,r_{1}})=n+\frac{n(n-2r_{1}-1)}{N}+m+\frac{m(m-2r_{2}-1)}{N}-\frac{(n-m)^{2}}{N^{2}},\]
\begin{multline*}
s_{\lambda^{N}_{m,r_{2},n,r_{1}}}(I_{N})=\frac{(N-r_{1}-r_{2}-1)(N+n+m-r_{1}-r_{2}-1)}{(N+n-r_{1}-r_{2}-1)(N+m-r_{1}-r_{2}-1)}\times \\
  \binom{n-1}{r_{1}}\binom{N+n-r_{1}-1}{n}
\binom{m-1}{r_{2}}\binom{N+m-r_{2}-1}{m}.
\end{multline*}
\end{lem}

\begin{proof} The general formulae are the following : for all $\alpha\in \N^{N}_{\downarrow}$, one has
\[c(\alpha)=\frac{1}{N}\left(\sum_{i=1}^{N}\alpha_{i}^{2} +\sum_{1\leq i<j\leq N} (\alpha_{i}-\alpha_{j}) \right)-\frac{1}{N^{2}} \left(\sum_{i=1}^{N}\alpha_{i} \right)^{2}\]
on one hand and, using the notation $\Delta(\lambda)=\prod_{1\leq i<j \leq N} (\lambda_{i}-\lambda_{j})$ and $\delta=(N-1,N-2,\ldots,1,0)$,
\[s_{\alpha}(I_{N})=\frac{\Delta(\alpha+\delta)}{\Delta(\alpha)}\]
on the other hand.
\end{proof}

Theorem \ref{covariance traces} now easily  follows from Proposition \ref{cov but c and s} and Lemma \ref{c and s}.

\def\cprime{$'$}

\end{document}